\documentclass[12pt]{amsart}
\usepackage{amssymb, latexsym, amsmath, amsfonts}
\usepackage{verbatim}
\newtheorem{thm}{Theorem}[section]
\newtheorem{cor}[thm]{Corollary}
\newtheorem{lem}[thm]{Lemma} 
\newtheorem{prop}[thm]{Proposition}
\theoremstyle{definition} 

\theoremstyle{remark}
\newtheorem{rem}[thm]{Remark}
\numberwithin{equation}{section}

\newcommand{\mbb}{\mathbb}
\newcommand{\ra}{\rightarrow}
\newcommand{\ze}{\zeta}
\newcommand{\pa}{\partial}
\newcommand{\ov}{\overline}
\newcommand{\sm}{\setminus}
\newcommand{\ep}{\epsilon}
\newcommand{\no}{\noindent}
\newcommand{\Om}{\Omega}
\newcommand{\cal}{\mathcal}
\newcommand{\ti}{\tilde}

\newcommand{\al}{\alpha}
\newcommand{\be}{\beta}
\newcommand{\ga}{\gamma}
\newcommand{\de}{\delta}

\newcommand{\De}{\Delta}
\newcommand{\Ga}{\Gamma}
\newcommand{\om}{\omega}

\begin{document}
\title{Some regularity theorems for CR mappings}
\thanks{The second author was supported in part by a grant from the UGC under DSA-SAP (Phase IV) and 
the DST SwarnaJayanti Fellowship 2009-2010}
\author{G. P. Balakumar}
\address{Department of Mathematics,
Indian Institute of Science, Bangalore 560 012, India}
\email{gpbalakumar@math.iisc.ernet.in}
\author{Kaushal Verma}
\address{Department of Mathematics,
Indian Institute of Science, Bangalore 560 012, India}
\email{kverma@math.iisc.ernet.in}

\begin{abstract}
The purpose of this article is to study Lipschitz CR mappings from an $h$-extendible (or semi-regular) hypersurface in $\mbb C^n$. Under various assumptions 
on the target hypersurface, it is 
shown that such mappings must be smooth. A rigidity result for proper holomorphic mappings from strongly pseudoconvex domains is also proved.
\end{abstract}

\maketitle

\section{Introduction}

\no The purpose of this note is to prove regularity results for Lipschitz CR mappings from $h$-extendible hypersurfaces in $\mbb C^n$. 

\begin{thm}
Let $f : M \ra M'$ be a non-constant Lipschitz CR mapping between $C^{\infty}$ smooth pseudoconvex finite type (in the sense of D'Angelo) hypersurfaces in $\mbb C^n$. 
Let $p \in M$ and $p'= f(p) \in M'$. Assume that $M$ is $h$-extendible at $p$, and that there is an open neighbourhood $U' \subset \mbb C^n$ of $p'$ and a 
$C^{\infty}$ smooth defining function $r'(z')$ for $U'\cap M'$ in $U'$ which is plurisubharmonic on the pseudoconvex side of $M'$ near $p'$.
If the Levi rank of $M'$ at $p'$ is at least $n - 2$ then $f$ is $C^{\infty}$ smooth in a neighbourhood of $p$. 
\end{thm}

\no This theorem, which is purely local in nature, is motivated by the regularity and rigidity results for CR mappings obtained in \cite{CS}, \cite{CGS}, \cite{CPS1} and \cite{G}. Let $B(p, \ep), B(p', \ep')$ be small open balls around $p, p'$ respectively. The pseudoconvex side of $M$ near $p$ is that component of $B(p, \ep) \sm M$ which is a
pseudoconvex domain for small $\ep > 0$. Denote this by $B^-(p, \ep)$ while the other component (the pseudoconcave part) will be denoted by $B^+(p, \ep)$. The same practice
will be followed while referring to the respective components of the complement of $M'$ in $B(p', \ep')$. Since $M$ and $M'$ are both pseudoconvex and of finite type near
$p$ and $p'$ respectively, it follows that $f$ admits a holomorphic extension to the pseudoconvex side of $M$ near $p$ and
this extension (which will still be denoted by $f$) maps the pseudoconvex side of $M$ near $p$ into the pseudoconvex side of $M'$ near $p'$. In fact, since $f : M \ra M'$ is 
assumed to be Lipschitz, it follows that the extension is also Lipschitz on $B^-(p, \ep) \cup M$ near $p$. Fix $\ep, \ep'> 0$ so that  
\[
f : B^-(p, \ep) \ra  B^-(p', \ep')
\]
is holomorphic, $f$ extends Lipschitz continuously up to $M \cap B(p, \ep)$ and satisfies $f(M \cap B(p, \ep)) \subset M' \cap B(p', \ep')$. For brevity, let $D$ and $D'$ denote 
the pseudoconvex domains $B^-(p, \ep)$ and $B^-(p', \ep')$ respectively so that $M \cap B(p, \ep)$ and $M' \cap B(p', \ep')$ are smooth open finite type pieces on their
boundaries. We may also assume that $\ep'> 0$ is small enough so that the defining function $r'(z') \in C^{\infty}(B(p', \ep'))$ and $D'= \{ r'(z') < 0 \}$ where $r'(z')$ is 
plurisubharmonic in $D'$. Note that while this condition holds for strongly pseudoconvex and convex finite type domains in $\mbb C^n$, not all (see \cite{DF}) 
pseudoconvex domains satisfy this condition. Now observe that the map $f : D \ra D'$ is not known to be proper though it does extend continuously up to $M \cap B(p, \ep)$. A 
sufficient condition for $f$ to be proper from $D$ onto $D'$ was considered by Bell-Catlin in \cite{BC} -- namely, if $f^{-1}(p') = f^{-1}(f(p))$ is compact in $M \cap B(p, \ep)$ 
then it is possible to choose $D, D'$ in such a way that $f$ extends as a proper holomorphic mapping between these domains. Once this is established, the $C^{\infty}$ 
smoothness of $f$ is a consequence of the techniques in \cite{BC}. The proof of theorem 1.1 therefore consists of showing that the fibre over $p'$ is compact in 
$M \cap B(p, \ep)$. To do this, we adapt the method of uniform scaling from \cite{CS},\cite{CGS} and \cite{CPS1}. The domains $D, D'$ and the map $f$ are scaled by composing $f$ 
with a suitable family of automorphisms of $\mbb C^n$ that enlarge these domains near $p, p'$ respectively. This produces a scaled family $f^{\nu} : D_{\nu} \ra D'_{\nu}$ 
of maps between scaled domains. The scaled domains $D_{\nu}$ converge in the Hausdorff sense to a model domain of the form
\begin{equation}
D_{\infty} = \big\{ z \in \mbb C^n : 2 \Re z_n + P('z, ' \ov z) < 0 \big\}
\end{equation}
where $'z = (z_1, z_2, \ldots, z_{n - 1})$ and $P('z, '\ov z)$ is a weighted homogeneous (the weights being determined by the Catlin multitype of $\pa D$ at $p$) real valued 
plurisubharmonic polynomial without pluriharmonic terms of total weight $1$, while the domains
$D'_{\nu}$ converge to 
\begin{equation}
D'_{\infty} =  \big\{ z \in \mbb C^n : 2 \Re z_n + Q_{2m'}(z_1, \ov z_1) + \vert z_2 \vert^2 + \ldots + \vert z_{n - 1} \vert^2 < 0 \big\}.
\end{equation}
Here $2m'$ is the type of $\pa D'$ at $p'$ and  $Q_{2m'}(z_1, \ov z_1)$ is a subharmonic polynomial of degree at most $2m'$ without harmonic terms. $D'_{\infty}$ is manifestly of 
finite type while $D_{\infty}$ also has the same property since $p$ is assumed to be $h$-extendible. The local geometry of a smooth pseudoconvex finite type hypersurface whose 
Levi rank is at least $n - 2$ guarantees the existence of special 
polydiscs around points sufficiently close to the hypersurface on which the defining function does not change by more than a prescribed 
amount -- these are the analogues of Catlin's bidiscs and have been considered earlier in \cite{Ch3}. The holomorphic mappings $f_{\nu} : D_{\nu} \ra D'_{\nu}$ are shown to form a 
normal family by using these polydiscs as in \cite{TT}; the limit map $F : D_{\infty} \ra D'_{\infty}$ is therefore holomorphic. The assumptions that $f$ is Lipschitz and that 
the defining function of $\pa D'$ is plurisubharmonic on $D'$ force $F$ to be non-degenerate. Furthermore, since $f$ extends Lipschitz continuously up to 
$M$ near $p$, it is natural to expect that $F$ imbibes some regularity near $0 \in \pa D_{\infty}$. This indeed happens and it is possible to show that $F$ is H\"{o}lder 
continuous up to $\pa D_{\infty}$ near the origin. The main ingredient needed to do this is a stable rate of blow up of the Kobayashi metric on $D'_{\nu}$ for all 
$\nu \gg 1$ and this follows by analyzing the behaviour of analytic discs in $D'_\nu$ whose centers lie close to the origin. In particular, theorem 2.3 provides a stable lower bound for the Kobayashi metric in $D'_\nu$ near the origin using ideas from \cite{TT}. Once $F$ is known to be H\"{o}lder continuous up to 
$\pa D_{\infty}$ near the origin, Webster's theorem (\cite{W}) implies that $F$ must be algebraic. Moreover, if $f$ has a noncompact fibre over $p'$, it can be shown that 
the same must hold for that of $F$ over $F(0)$. This violates the invariance property of Segre varieties associated to $\pa D_{\infty}$ and $\pa D'_{\infty}$ and proves the 
compactness of the fibre of $f$ over $p' \in \pa D'$. 

\medskip

It is known that smooth convex finite type hypersurfaces can also be scaled and morever the stability of the Kobayashi metric on the scaled domains is also understood (see 
\cite{CGS}). Hence the same line of reasoning yields the following:

\begin{cor}
With $M$ as in theorem 1.1, let $M' \subset \mbb C^n$ be a smooth convex finite type hypersurface and $f : M \ra M'$ a Lipschitz CR mapping. As before, let $p \in M$ and $p'= 
f(p) \in M'$. Then $f$ is $C^{\infty}$ smooth in a neighbourhood of $p$.
\end{cor}

\no More can be said when at least one of the hypersurfaces is strongly pseudoconvex. We first consider the following local situation -- let $M \subset \mbb C^n$ be a 
$C^{\infty}$ smooth pseudoconvex finite type hypersurface and let $p \in M$ be an $h$-extendible point. Let the Catlin multitype of $M$ at $p$ be $(1, m_2, \ldots, m_n)$ 
where the $m_i$'s form an increasing sequence of even integers. Then there exists a holomorphic coordinate system around $p = 0$ in which the defining function for $M$ takes 
the form:
\[
r(z) = 2 \Re z_n + P('z, '\ov z) + R(z)
\]
where $P('z, '\ov z)$ is a $(1/{m_n}, 1/m_{n - 1}, \ldots, 1/m_2)$-homogeneous plurisubharmonic polynomial of total weight one without pluriharmonic terms and the error $R(z)$ 
has weight strictly bigger than one. As usual, if $J = (j_1, j_2, \ldots, j_{n - 1})$ is a multiindex of length $n - 1$, then $'z^J$ denotes the monomial $z_1^{j_1} z_2^{j_2} \cdots 
z_{n - 1}^{j_{n - 1}}$ and $'\ov z^J = \ov z_1^{j_1} \ov z_2^{j_2} \cdots \ov z_{n - 1}^{j_{n - 1}}$.  

\begin{thm}
Let $f : M \ra M'$ be a nonconstant Lipschitz CR mapping between real hypersurfaces in $\mbb C^n$. Fix $p \in M$ and $p' = f(p) \in M'$. Suppose that $M$ is $C^{\infty}$ smooth 
pseudoconvex and of finite type near $p$ and that $M'$ is $C^2$ strongly pseudoconvex near $p'$. If $M$ is $h$-extendible at $p$, then the weighted homogeneous polynomial in the 
defining function for $M$ near $p = 0$ can be expressed as
\[
P('z, '\ov z) = \vert P_1('z) \vert^2 + \vert P_2('z) \vert^2 + \ldots + \vert P_{n - 1}('z) \vert^2
\]
where each $P_i('z)$ for $1 \le i \le n - 1$ is a weighted holomorphic polynomial of total weight $1/2$. Moreover, the algebraic variety 
\[
V= \big\{ 'z \in \mbb C^{n - 1} : P_1('z) = P_2('z) = \ldots = P_{n - 1}('z) = 0 \big\}
\]
contains $'0 \in \mbb C^{n - 1}$ as an isolated point. In particular, there exist constants $c_j > 0$ for $1 \le j \le n - 1$ such that
\[
P('z,'\bar{z}) = c_1 \vert z_1 \vert^{m_n} + c_2 \vert z_2 \vert^{m_{n - 1}} + \ldots + c_{n - 1} \vert z_{n - 1} \vert^{m_2} + {\text mixed \; terms}
\]
where the phrase `mixed terms' denotes a sum of weight one monomials annihilated by at least one of the natural quotient maps $\mbb C['z, '\bar{z}] \ra \mbb 
C['z, '\bar{z}]/(z_j \bar{z}_k)$ for $1 \le j, k \le n-1$, $j \not= k$.
\end{thm}

\no As in theorem 1.1, the limit map $F : D_{\infty} \ra D'_{\infty}$ is holomorphic, nonconstant and hence algebraic. Since $D'_{\infty} \backsimeq \mbb B^n$, it is possible to 
show that $F$ extends holomorphically across $0 \in \pa D_{\infty}$ and $F(0) = 0'\in \pa D'_{\infty}$. The explicit description of the weighted homogeneous polynomial 
$P('z, '\ov z)$ follows from working with the extended mapping near the origin. On the other hand, it is also natural to consider the case of CR mappings from strongly pseudoconvex 
hypersurfaces and we have both a global and a local version -- in the spirit of theorem 1.1, for this case. First recall that a domain $D \subset \mbb C^n$ is said to be regular 
at $p \in \pa D$ if there is a pair of open neighbourhoods $V \subset U$ of $p$, constants $M > 0$, $0 < \al \le 1$ and $\be > 1$ such that for any $\ze \in V \cap \pa D$, there 
is a function $\phi_{\ze}$  which is continuous on $U \cap \ov D$, plurisubharmonic on $U \cap D$ and satisfies
\[
-M \vert z -  \ze \vert^{\al} \le \phi_{\ze}(z) \le - \vert z - \ze \vert^{\be}
\]
for all $z \in U \cap \ov D$. It is known that the class of regular points includes open pieces of strongly pseudoconvex boundaries, smooth weakly pseudoconvex finite type pieces 
in $\mbb C^2$, those that are smooth convex finite type in $\mbb C^n$ (see \cite{FS}) and finally smooth pseudoconvex finite type boundaries in $\mbb C^n$ (see \cite{Ch2}). $D$ 
is said to be regular if each of its boundary points is regular.

\begin{thm}
Let $D \subset \mbb C^n$ be a bounded regular domain, $D' \subset \mbb C^n$ a possibly unbounded domain and $f : D \ra D'$ a proper holomorphic mapping. Let $p \in \pa D$ be a 
$C^2$ strongly pseudoconvex point and $p'\in \pa D'$ be such that the boundary $\pa D'$ is $C^{\infty}$ smooth pseudoconvex and of finite type near $p'$. Suppose that the Levi 
rank of $\pa D'$ at $p'$ is $n - 2$ and assume that $p' \in cl_f(p)$, the cluster set of $p$. Then $p'$ is also a strongly pseudoconvex point.
\end{thm}

\no This fits in the paradigm, observed earlier by many other authors (for example \cite{DF1}), that a proper mapping does not increase the type of a boundary point. Note 
furthermore that there are no other assumptions on the defining function for $\pa D'$ near $p'$ as 
in theorem 1.1 except the Levi rank condition. The difficulty created by the lack of this assumption as explained above is circumvented by the global properness of $f$ -- indeed, 
it is possible to scale $f$ to get a holomorphic limit $F : \mbb B^n \ra D'_{\infty}$ where $D'_{\infty}$ is as in (1.2). Now using the fact that $p \in \pa D$ is strongly 
pseudoconvex and hence regular, it is possible to peel off a local correspondence from the global one $f^{-1} : D' \ra D$ that extends continuously up to $\pa D'$ near $p'$ and 
contains $p$ in its cluster set by \cite{BS}. This local correspondence can be scaled, using the Schwarz lemma for correspondences from \cite{V} and the behaviour of the 
scaled balls in the Kobayashi metric from \cite{MV}. This gives a well defined correspondence from $D'_{\infty}$ with values in $\mbb B^n$ and this turns out to be the inverse 
for $F$. Thus $F : \mbb B^n \ra D'_{\infty}$ is proper and this is sufficient to conclude that $p' \in \pa D'$ must be strongly pseudoconvex. 

\begin{cor}
Let $f : M \ra M'$ be a nonconstant Lipschitz CR mapping between real hypersurfaces in $\mbb C^n$. Fix $p \in M$ and $p' = f(p) \in M'$. Suppose that $M$ is $C^2$ strongly 
pseudoconvex near $p$ and that there is an open neighbourhood $U' \subset \mbb C^n$ of $p'$ and a $C^{\infty}$ smooth defining function $r'(z')$ for $U'\cap M'$ in $U'$ which 
is plurisubharmonic on the pseudoconvex side of $M'$ near $p'$. If the Levi rank of $M'$ at $p'$ is at least $n - 2$ then $p'$ is a strongly pseudoconvex point.
\end{cor}

The first author would like to thank Herv\'{e} Gaussier for very patiently listening to the material presented here.


\section{Proof of theorem 1.1}

\no Consider a smooth pseudoconvex finite type hypersurface $M \subset \mbb C^n$. Associated to each $p \in M$ are two well known invariants: one is the D'Angelo type 
\[
\De(p) = (\De_n(p), \De_{n - 1}(p), \ldots, \De_1(p))
\]
where the integer $\De_q(p)$ is the $q$-type of $M$ at $p$ and is a measure of the maximal order of contact of $q$ dimensional 
varieties with $M$ at $p$. To recall the definition (see \cite{D}), let $r$ be a local defining function for $M$ near $p$ and let $\ti r(z) = r(z + p)$. Then for $1 \le q \le n$,
\[
\De_q(p) = \inf_L \sup_\tau \{ \nu(\ti r \circ L \circ \tau) / \nu(\tau) \}
\]
where the infimum is taken over all linear embeddings $L : \mbb C^{n - q + 1} \ra \mbb C^n$ and the supremum is taken over all germs of holomorphic curves $\tau : (\mbb C, 0) 
\ra (\mbb C^{n - q + 1}, 0)$ mapping the origin in $\mbb C$ to the origin in $\mbb C^{n - q + 1}$ and $\nu(f)$ denotes the order of vanishing of $f$ at the origin. The 
smoothness of $M$ at $p$ implies that $\De_n(p) = 1$ and it can be seen that $2 \le \De_{n - 1}(p) \le \De_{n - 2}(p) \le \ldots \le \De_1(p) < \infty$.

\medskip

To quickly recall the Catlin multitype of $M$ at $p$ (see \cite{C3}), let $\Ga_n$ be the collection of $n$-tuples of reals $m = (m_1, m_2, \ldots, m_n)$ such that $0 < m_1 \le 
m_2 \le \ldots \le m_n \le \infty$. Order $\Ga_n$ lexicographically. An element $m$ of $\Ga_n$ is called distinguished provided there is a holomorphic coordinate system $w = 
\phi(z)$ around $p$ with mapped to the origin such that if
\[
(\al_1 + \be_1)/m_1 + (\al_2 + \be_2)/m_2 + \ldots + (\al_n + \be_n)/m_n < 1
\]
then $D^{\al} \ov D^{\be} r \circ \phi^{-1}(0) = 0$; here $\al = (\al_1, \al_2, \ldots, \al_n)$ and $\be = (\be_1, \be_2, \ldots, \be_n)$ are $n$-tuples and $D^{\al}$ and $\ov 
D^{\be}$ are the partial derivatives
\[
\pa^{\vert \al \vert} / \pa z_1^{\al_1} \pa z_2^{\al_2} \ldots \pa z_n^{\al_n}  \; {\rm and} \;  \pa^{\vert \be \vert} / \pa \bar{z}_1^{\be_1}  \pa \bar{z}_2^{\be_2} \ldots  \pa \bar{z}_n^{\be_n} 
\]
respectively. The Catlin multitype $\cal M(p) = (m_1, m_2, \ldots, m_n)$ of $M$ at $p$ is defined to be the largest amongst all distinguished elements. Since $r(z)$ is a smooth 
defining function for $M$ near $p$, it can be seen that the first entry in $\cal M(p)$ is always one. If $\cal M(p)$ is finite, i.e., $m_n < \infty$, then there is a coordinate 
system around $p = 0$ such that the defining function is of the form
\begin{equation}
r(z) = 2 \Re z_n + P('z, ' \ov z) + R(z)
\end{equation}
where $P('z, ' \ov z)$ is a $(1/{m_n}, 1/{m_{n - 1}}, \ldots, 1/m_2)$ homogeneous polynomial of total weight one which is plurisubharmonic and does not contain pluriharmonic terms 
and 
\begin{equation}
\vert R(z) \vert \lesssim \big( \vert z_1 \vert^{m_n} + \vert z_2 \vert^{m_{n - 1}} + \ldots + \vert z_n \vert^{m_1} \big)^{\ga}
\end{equation}
for some $\ga > 1$. The homogeneity of $P('z, '\ov z)$ with total weight one mentioned above means that for $t \ge 0$, $P \circ \pi_t('z) = t P('z, '\ov z)$ where
\[
\pi_t(z_1, z_2, \ldots, z_{n - 1}) = (t^{1/{m_n}} z_1, t^{1/{m_{n - 1}}} z_2, \ldots, t^{1/{m_2}} z_{n - 1}). 
\]
Other total weights $\mu > 0$ occur when $t^{\mu}$ (instead of $t$) can be factored from $P \circ \pi_t('z)$. The plurisubharmonicity of $P('z, '\ov z)$ implies that each $m_k$ 
for $2 \le k \le n$ must be even. Thus the variable $z_i$ is assigned a weight of $1/{m_{n - i + 1}}$ for $1 \le i \le n - 1$ and by definition the weight of the 
monomial ${'z^J}{'\ov z^K}$ is 
\[
(j_1 + k_1)/m_n + (j_2 + k_2)/{m_{n - 1}} + \ldots + (j_{n - 1} + k_{n - 1})/{m_2}
\]
for $(n - 1)$-multiindices $J = (j_1, j_2, \ldots, j_{n-1})$ and $K = (k_1, k_2, \ldots, k_{n - 1})$. A basic relation between $\De(p)$ and $\mathcal M (p)$ proved in 
\cite{C3} is that $m_{n + 1 - q} \le \De_q$ for all $1 \le q \le n$. 

\medskip

Call $p \in M$ an $h$-extendible (or semi-regular) point if $\cal M (p)$ is finite and $\De(p) = \cal M (p)$. This happens if and only if (see \cite{Yu1}) there is a $(1/m_n, 
1/m_{n - 1}, \ldots, 1/m_2)$ homogeneous $C^1$ smooth real function $a('z)$ on $\mbb C^{n - 1} \sm \{0\}$ such that $a('z) > 0$ whenever $'z \not= 0$ and $P('z, '\ov z) - 
a('z)$ is plurisubharmonic on $\mbb C^{n - 1}$ and among other things, this is equivalent to the model domain $D_{\infty}$ (as in (1.1)) being of finite type. We shall henceforth 
assume that $p = 0$ and $p'= 0'$ and that the respective defining functions $r(z)$ and $r'(z')$ satisfy $\pa r / \pa z_n (0) \not= 0$ and $\pa r'/ \pa z'_n (0') \not= 0$.
The holomorphic map $f : D \ra D'$ is not necessarily proper, but H\"{o}lder continuous with exponent $\de \in (0, 1)$ on $\ov D$ near $0 \in \pa D$ by \cite{Su} and $f(0) = 0'$. 
Thus the assumption that $f$ is Lipschitz is stronger than what is apriori known.


\subsection{The Scaling Method applied to $(D, D', f)$} For $z \in D$ close to the origin, note that
\[
{\rm dist}(z, \pa D) \; \lesssim \; {\rm dist}(f(z), \pa D') \; \lesssim \; {\rm dist}(z, \pa D)
\]
where the inequality on the right follows since $f$ admits a Lipschitz extension to $D$ near the origin, while the left inequality follows by applying the Hopf lemma to $r'\circ 
f(z)$ which is a negative plurisubharmonic function on $D$. To scale $D$, choose a sequence of points $p^{\nu} = ('0, -\de_{\nu})$ in $D$ along the inner normal at the origin, where 
$\de_{\nu} > 0$ and $\de_{\nu} \searrow 0$. Let $T^{\nu}$ be the dilation defined by
\[
T^{\nu} : (z_1, z_2, \ldots, z_{n - 1}, z_n) \mapsto (\de_{\nu}^{-1/{m_n}} z_1, \de_{\nu}^{-1/{m_{n - 1}}} z_2, \ldots, \de_{\nu}^{-1/{m_2}} z_{n - 1}, \de_{\nu}^{-1} z_n)
\]
and note that $T^{\nu}(p^{\nu}) = ('0, -1)$ while the domains $D_{\nu} = T^{\nu}(D)$ are defined by
\[
r_{\nu} = \de_{\nu}^{-1} \; r \circ {(T^{\nu})}^{-1}(z) = 2 \Re z_n + P('z, '\ov z) + \de_{\nu}^{-1} R \circ {(T^{\nu})}^{-1}(z)
\]
where
\[
\big\vert  \de_{\nu}^{-1} \; R \circ {(T^{\nu})}^{-1}(z) \big\vert \; \lesssim \; \de_{\nu}^{\ga - 1} \big( \vert z_1 \vert^{m_n} + \vert z_2 \vert^{m_{n - 1}} + \ldots + \vert 
z_{n - 1} \vert^{m_2} + \vert z_n \vert \big)^{\ga}
\]
by (2.1). On each compact set in $\mbb C^n$ this error term converges to zero since $\ga > 1$ and hence the sequence of domains $D_{\nu}$ converges in the Hausdorff metric to
\[
D_{\infty} =  \big\{ z \in \mbb C^n : 2 \Re z_n + P('z, ' \ov z) < 0 \big\}.
\]
Let $r_{\infty}(z) =  2 \Re z_n + P('z, ' \ov z)$. To scale $D'$ recall that by \cite{Ch3}, for each $\ze$ near $0' \in \pa D'$ there is a unique polynomial automorphism 
$\Phi_{\ze}(z) : \mbb C^n \ra \mbb C^n$ with $\Phi_{\ze}(\ze) = 0$ such that
\begin{multline*}
r (\Phi_{\ze}^{-1}(z)) = r(\ze) + 2 \Re z_n + \sum_ {\substack{ j + k \le 2m\\
								j, k > 0}} a_{jk}(\ze) z_1^j \ov z_1^k + \sum_{\al = 1}^{n - 1} \vert z_{\al} \vert^2 + 
		\sum_{\al = 2}^{n - 1} \sum_{ \substack{j + k \le m\\
                                                          j, k > 0}} \Re \Big( \big(b_{jk}^{\al}(\ze) z_1^j \ov z_1^k \big) z_{\al} \Big) \\
+ O \big( \vert z_n \vert \vert z \vert + \vert z_{\ast} \vert^2 \vert z \vert + \vert z_{\ast} \vert \vert z_1 \vert^{m + 1} + \vert z_1 \vert^{2m + 1} \big).
\end{multline*}
where for $z=(z_1,\ldots,z_n) \in \mathbb{C}^n$, we denote $z_* = (z_2,\ldots,z_{n-1})\in \mathbb{C}^{n-2}$.
These automorphisms converge to the identity uniformly on compact subsets of $\mbb C^n$ as $\ze \ra 0$. Furthermore, if for $\ze = (\ze_1, \ze_2, \ldots, \ze_n) \in D'$ as above 
we consider the point $\ti \ze = (\ze_1, \ze_2, \ldots, \ze_n + \ep)$ where $\ep > 0$ is chosen to ensure that $\ti \ze \in \pa D'$, then the actual form of $\Phi_{\ze}(z)$ shows 
that 
$\Phi_{\ti \ze}(\ze) = (0, \ldots,0, -\ep)$ -- since the explicit description of these automorphisms will come up later, we shall be content at this stage with merely 
collecting the relevant properties needed to describe the scaling of $D'$. To define the distinguished polydiscs around $\ze$ (more precisely, biholomorphic images of polydiscs), 
let
\[
A_l(\ze) = \max \{ \vert a_{jk}(\ze) \vert : j + k = l\}, \;\; 2 \le l \le 2m'
\]
and 
\[
B_{l'}(\ze) = \max \{ \vert b^{\al}_{jk}(\ze) \vert : j + k = l', 2 \le \al \le n - 1\}, \; \; 2 \le l'\le m'.
\]  
For each $\de > 0$ define
\[
\tau(\ze, \de) = \min \big\{ \big( \de / A_l(\ze) \big)^{1/l}, \big( \de^{1/2}/ B_{l'}(\ze) \big)^{1/l'} : \; 2 \le l \le 2m', \; 2 \le l'\le m' \big\}
\]
Since the type of $\pa D'$ at the origin is $2m'$ it follows that $A_{2m'}(0) > 0$ and hence $A_{2m'}(\ze)$ is positive for all $\ze$ sufficiently close to the origin. Thus
\[
\de^{1/2} \lesssim \tau(\ze, \de) \lesssim \de^{1/2m'}
\]
for $\ze$ close to the origin -- the upper bound being a consequence of the non-vanishing of $A_{2m'}(\ze)$ near the origin while the lower bound follows since the greatest 
possible exponent of $\de$ in the definition of $\tau(\ze, \de)$ is $1/2$. Set 
\[
\tau_1(\ze, \de) = \tau(\ze, \de) = \tau, \tau_2(\ze, \de) = \ldots = \tau_{n - 1}(\ze, \de) = \de^{1/2}, \tau_n(\ze, \de) = \de
\]
and define
\[
R(\ze, \de) = \{z \in \mbb C^n : \vert z_k \vert < \tau_k(\ze, \de), 1 \le k \le n\}
\]
which is a polydisc around the origin in $\mbb C^n$ with polyradii $\tau_k(\ze, \de)$ along the $z_k$ direction for $1 \le k \le n$ and let
\[
Q(\ze, \de) = \Phi_{\ze}^{-1}\big( R(\ze, \de) \big)
\]
which is a distorted polydisc around $\ze$. It was shown in \cite{TT} that these domains satisfy the engulfing property, i.e., for all $\ze$ in a small fixed neighbourhood of the 
origin, there is a uniform constant $C > 0$ such that if $\eta \in Q(\ze, \de)$, then $Q(\eta, \de) \subset Q(\ze, C \de)$ and $Q(\ze, \de) \subset Q(\eta, C \de)$.

\medskip

Consider the sequence $p'^{\nu} = f(p^{\nu}) \in D'$ that converges to the origin and denote by $w'^{\nu}$ the point on $\pa D'$ chosen such that if $p'^{\nu} = (p'^{\nu}_1, 
p'^{\nu}_2, \ldots, p'^{\nu}_n)$ then $w'^{\nu} = (p'^{\nu}_1, p'^{\nu}_2, \ldots, p'^{\nu}_n + \ga_{\nu})$ for some $\ga_{\nu} > 0$. Note that
\[
\ga_{\nu} \approx {\rm dist}(p'^{\nu}, \pa D')
\]
for all large $\nu$. Hence
\[
\de_{\nu} = {\rm dist}(p^{\nu}, \pa D) \approx {\rm dist}(p'^{\nu}, \pa D') \approx \ga_{\nu}
\]
for all large $\nu$. Let $g^{\nu} = \Phi_{w'^{\nu}}(\cdot)$ be the polynomial automorphism of $\mbb C^n$ corresponding to $w'^{\nu} \in D'$ as described above. Let us consider the 
holomorphic mappings
\[
f^{\nu} = g^{\nu} \circ f : D \ra g^{\nu}(D')
\]
and define a dilation of coordinates in the target space by 
\[
B^{\nu} : (z'_1, z'_2, \ldots, z'_n) \mapsto \big( (\tau^{\nu}_1)^{-1} z'_1, (\tau^{\nu}_2)^{-1} z'_2, \ldots, (\tau^{\nu}_n)^{-1} z'_n \big)
\]
where $\tau^{\nu}_1 = \tau(w'^{\nu}, \ga_{\nu})$, $\tau^{\nu}_j = \ga^{\nu}_{1/2}$ for $2 \le j \le n - 1$ and $\tau^{\nu}_n = \ga_{\nu}$. Let $D^{\nu} = T^{\nu}(D)$ and 
$D'^{\nu}= (B^{\nu} \circ g^{\nu})(D')$ be the scaled domains and the scaled maps between them are 
\[
F^{\nu} = B^{\nu} \circ f^{\nu} \circ (T^{\nu})^{-1} : D_{\nu} \ra D'_{\nu}.
\]
To understand the Hausdorff limit of the domains $D'_{\nu}$, note first that $B^{\nu} \circ g^{\nu}(p'^{\nu}) = ('0, -1)$, which implies that $F^{\nu}('0, -1) = ('0, -1)$ for 
all $\nu$, and that $r'_{\nu}$, the defining function for $D'_{\nu}$ is given by
\[
\ga_{\nu}^{-1} \; r'\circ (B^{\nu} \circ g^{\nu})^{-1}(z) = 2 \Re z_n + Q_{\nu}(z_1, \ov z_1) + \sum_{\al = 2}^{n - 1} \vert z_{\al} \vert^2 + \sum_{\al = 2}^{n - 1} \Re \big( 
S^{\al}_{\nu}(z_1, \ov z_1) z_{\al} \big) + O(\tau^{\nu}_1)
\]
where
\[
Q_{\nu}(z_1, \ov z_1) =  \sum_ {\substack{ j + k \le 2m'\\
                                                  j, k > 0}} a_{jk}(w'^{\nu}) \ga_{\nu}^{-1} (\tau^{\nu}_1)^{j + k} z_1^j \ov z_1^k
\]
and
\[
S^{\al}_{\nu}(z_1, \ov z_1) =  \sum_ {\substack{ j + k \le m'\\
                                                        j, k > 0}} b^{\al}_{jk}(w'^{\nu}) \ga_{\nu}^{-1/2} (\tau^{\nu}_1)^{j + k} z_1^j \ov z_1^k.
\]
By the definition of $A_l, B_{l'}$ and $\tau^{\nu}_1$ it follows that the largest coefficient in both $Q_{\nu}$ and $S^{\al}_{\nu}$ is at most one in modulus. It was shown in 
\cite{Ch3} that there exists a uniform $\ep > 0$ such that
\[
\big \vert b^{\al}_{jk}(w'^{\nu}) \ga_{\nu}^{-1/2} (\tau^{\nu}_1)^{j + k} \big \vert \lesssim (\tau^{\nu}_1)^{\ep}
\] 
for all possible indices $j, k, \al$ and all large $\nu$. Therefore some subsequence of this family of defining functions converges together with all derivatives on 
compact sets to
\[
r'_{\infty}(z) = 2 \Re z_n + Q_{2m'}(z_1, \ov z_1) + \vert z_2 \vert^2 + \ldots + \vert z_{n - 1} \vert^2  
\]
where $Q_{2m'}(z_1, \ov z_1)$ is a polynomial of degree at most $2m'$ without harmonic terms. Hence the domains $D'_{\nu}$ converge to
\[
D'_{\infty} = \big \{ z \in \mbb C^n : 2 \Re z_n +  Q_{2m'}(z_1, \ov z_1) + \vert z_2 \vert^2 + \ldots + \vert z_{n - 1} \vert^2 < 0  \big \}
\]
which, being the smooth limit of pseudoconvex domains, is itself pseudoconvex. In particular, it follows that $Q_{2m'}(z_1, \ov z_1)$ is subharmonic.

\medskip

It is known that analytic discs in a bounded regular domain in $\mbb C^n$ satisfy the so-called attraction property (see \cite{Ber1}), i.e., if the centre of a given disc is 
close to a boundary point, then a given subdisc around the origin cannot wander too far away from the same boundary point. A quantitative version of this was proved by 
Berteloot-Coeur\'{e} in \cite{BerC} and forms the basis for controlling families of scaled mappings in $\mbb C^2$. Using ideas from \cite{Ber2} and \cite{Ch3}, the following 
analogue was proved in \cite{TT} and will be useful in this situation as well -- we include the statement for the sake of completeness.

\begin{prop}
Let $\Om \subset \mbb C^n$ be a bounded domain. Let $\ze_0 \in \pa \Om$ and suppose there is an open neighbourhood $V$ such that $V \cap \pa \Omega$ is $C^{\infty}$ smooth 
pseudoconvex of finite type and the Levi rank is at least $n - 2$ on $V \cap \pa \Om$. Fix a domain $\om \subset \mbb C^m$. 

Then for any fixed point $z_0 \in \om$ and a compact $K \subset \om$ containing $z_0$, there exist constants $\ep(K), C(K) > 0$ such that for any $\xi \in V \cap \pa \Om$ and $0 
< \ep < \ep(K)$, every holomorphic mapping $F : \om \ra \Om$ with 
\[
\vert F(z_0) - \ze_0 \vert < \ep(K) \;\; {\rm and} \;\; F(z_0) \in Q(\xi, \ep)
\]
also satisfies $F(K) \subset Q(\xi, C(K), \ep)$. 
\end{prop}

\no Now let $\{ K_j \}$ be an increasing sequence of relatively compact domains that exhaust $D_{\infty}$ such that each contains the base point $('0, -1)$. Fix $K = K_{\mu}$ and 
let ${\ti f}^{\nu} = f \circ (T^{\nu})^{-1}$. Then
\[
{\ti f}^{\nu}('0, -1) = f('0, -\de_{\nu}) = f(p^{\nu}) = p'^{\nu}
\]
and hence ${\ti f}^{\nu}('0, -1) \ra 0 \in \pa D'$. In particular, ${\ti f}^{\nu}('0, -1) = p'^{\nu} \in Q(w'^{\nu}, 2 \ga_{\nu})$ for all large $\nu$, by the construction of 
these distorted polydiscs. By the previous proposition
\[
{\ti f}^{\nu} (K) \subset Q(w'^{\nu}, C(K) \ga_{\nu})
\]
and therefore
\[
F^{\nu}(K) = B^{\nu} \circ g^{\nu} \circ {\ti f}^{\nu} (K) \subset B^{\nu} \circ g^{\nu}(Q(w'^{\nu}, C(K) \ga_{\nu}).
\]
However,
\[
B^{\nu} \circ g^{\nu}(Q(w'^{\nu}, C(K) \ga_{\nu}) = B^{\nu} (R(w'^{\nu}, C(K) \ga_{\nu}))
\]
which by definition is contained in a polydisk around the origin with polyradii 
\[
r_k = \tau_k(w'^{\nu}, C(K) \ga_{\nu}) / \tau_k(w'^{\nu}, \ga_{\nu})
\]
for $1 \le k \le n$. Note that $r_n = C(K)$ and that $r_k = (C(K))^{1/2}$ for $2 \le k \le n - 1$. Since $C(K) > 1$, and this may be assumed without loss of 
generality, the definition of $\tau(\ze, \de)$ shows that for $\de'< \de''$,
\[
(\de'/ \de'')^{1/2} \tau(\ze, \de'') \le \tau(\ze, \de') \le (\de' / \de'')^{1/2m'} \tau(\ze, \de'').
\]
Therefore
\[
\tau_1(w'^{\nu}, C(K) \ga_{\nu}) / \tau_1(w'^{\nu}, \ga_{\nu}) \le (C(K))^{1/2}
\]
and hence $r_1 \le (C(K))^{1/2}$. Thus $\{ F^{\nu} \}$ is uniformly bounded on each compact set in $D_{\infty}$ and is therefore normal. Let $F : D_{\infty} \ra D'_{\infty}$ be a 
holomorphic limit of some subsequence in $\{ F^{\nu} \}$ and since $F^{\nu}('0, -1) = ('0, -1)$ for all $\nu$ by construction, it follows that $F('0, -1) = ('0, -1)$. The maximum 
principle shows that $F(D_{\infty}) \subset D'_{\infty}$.

\medskip

Let $R > 0$ be arbitrary and fix $z \in D_{\infty} \cap B(0, R)$. Note that since $f$ preserves the distance to the boundary, it follows that
\[
\big \vert (r'_{\nu} \circ F^{\nu})(z) \big \vert = \ga_{\nu}^{-1} \big \vert (r'\circ f \circ T_{\nu}^{-1})(z) \big \vert \lesssim (\de_{\nu} / \ga_{\nu}) \big \vert 
r_{\nu}(z) \big \vert
\]
and hence
\[
\big \vert r'_{\infty} \circ F(z) \big \vert \lesssim \big \vert r_{\infty}(z) \big \vert
\]
since the constants appearing in the first estimate are independent of $\nu$. This shows that the cluster set of a finite boundary point of $D_{\infty}$ does not intersect 
$D'_{\infty}$ and therefore $F$ is non-degenerate. 

\subsection{Boundedness of $F^\nu(0)$}

A crucial point in the proof of theorem 1.1 is to show that near the origin, the family $\{ F^\nu \}$ of the scaled maps is uniformly H\"{o}lder continuous upto the 
boundary i.e., the H\"{o}lder constant and exponent are stable under the scaling. A first step in establishing this is to show that the images of the origin under these scaled 
maps is stable. 

\begin{prop}
The sequence $\{F^\nu(0)\}$ is bounded.
\end{prop}
\noindent {\it Proof.} Recall the scaled maps
\[
F^\nu = B^\nu \circ g^\nu \circ f \circ (T^\nu)^{-1}
\]
where $B^\nu,T^\nu$ were linear maps and $f(0)=0$. So $F^\nu(0)= (B^\nu \circ g^\nu \circ f)(0)=(B^\nu \circ g^\nu)(0)$. Now let us recall the explicit form of the map $g^\nu = \Phi_{w'^\nu}$ from \cite{TT}, which is a polynomial automorphism of $\mathbb{C}^n$ that reduces the defining function to a certain normal form
as stated in section 2.1. Let us recall this reduction procedure for any given domain $\Omega$ in $\mathbb{C}^n$ that is smooth pseudoconvex of finite type $2m$ and of Levi rank (at least) $n-2$ on a (small) open boundary piece $\Sigma$ in $\partial \Omega$ and containing the origin, say. Let $r$ be a smooth defining function for $\Sigma$ with $\partial r /\partial z_n(z) \ne 0$ for all $z$ in a small neighbourhood $U$ in $\mathbb{C}^n$ of $\Sigma$ such that the vector fields
\begin{equation*} 
L_n=\partial /\partial z_n,\ L_j=\partial / \partial z_j + b_j(z,\bar{z})\partial / \partial z_n
\end{equation*}
where $b_j = \big( \partial r/ \partial z_n \big)^{-1} \partial r / \partial z_j $, form a basis of $\mathbb{C}T^{(1,0)}(U)$ and satisfy $L_j r\equiv 0$ for $1 \leq j\leq n-1$ and 
$\partial\bar{\partial} r(z) (L_i,\bar{L}_j)_{2\leq i,j \leq n-1}$
has all its eigenvalues positive for each $z \in U$.
\medskip\\
The reduction now consists of five steps -- for $\zeta\in U$, the map $\Phi_\zeta = \phi_5 \circ \phi_4 \circ \phi_3 \circ \phi_2 \circ \phi_1 $ where each $\phi_j$ is described below.\medskip \\
$\phi_1$ is the affine map given by
\begin{equation*}
\phi_1(z_1,\ldots,z_n) = (z_1 - \zeta_1, \ldots, z_{n-1}-\zeta_{n-1},z_n+ \sum\limits_{j=1}^{n-1}b_j^\zeta z_j
-( \zeta_n + \sum\limits_{j=1}^{n-1} b_j^\zeta \zeta _j )).
\end{equation*}
where the coefficients $b_j^\zeta = b_j(\zeta,\bar{\zeta})$ are clearly smooth functions of $\zeta$ on $U$. Therefore, $\phi_1$ translates $\zeta$ to the origin and 
\[
r(\phi^{-1}(z))=r(\zeta)+ \Re z_n+ \text{terms of higher order}.
\]
where the constant term disappears when $\zeta\in \Sigma$.\medskip\\
The remaining reductions remove the occurrence of harmonic (not just pluriharmonic) monomials in the variables $z_1, \ldots ,z_{n-1}$ of weights upto one in the weighted homogeneous expansion of the defining function with respect to the weights given by the multitype for $\Sigma$, i.e., the variable $z_1$ is assigned a weight of $1/2m$, $z_n$ a weight of $1$ while the others are assigned $1/2$ each. 
Now, since the Levi form restricted to the subspace 
\[ 
L_*=\textrm{span}_{\mathbb{C}^n} \left< L_2,\ldots ,L_{n-1} \right>
\] 
of $T^{(1,0)}_{\zeta}(\partial \Omega)$ is positive definite, we may diagonalize it via a unitary transform $\phi_2$ and a dilation $\phi_3$ will then ensure that the quadratic part involving only $z_2,z_3,\ldots,z_{n-2}$ in the Taylor expansion of $r$ is $\vert z_2 \vert^2 + \vert z_3 \vert^2 + \ldots + \vert z_{n-2}\vert^2$. The entries of the matrix that represents the composite of the last two linear transformations are smooth functions of $\zeta$ and in the new coordinates still denoted by $z_1,\ldots z_n$, the defining function is in the form
\begin{multline}
r(z)=r(\zeta) + \Re z_n + \sum\limits_{\alpha=2}^{n-1}\sum\limits_{j=1}^{m}\Re\big(( a_j^\alpha z_1^j + b_j^\alpha \bar{z}_1^j )z_\alpha \big)+ \Re\sum\limits_{\alpha=2}^{n-1}c_\alpha z_\alpha^2 \\ 
+ \sum\limits_{2\leq j+k \leq 2m}a_{j,k}z_1^j \bar{z}_1^k + \sum\limits_{\alpha=2}^{n-1}\vert z_\alpha \vert^2 
+ \sum\limits_{\alpha=2}^{n-1}\sum_{\substack{j+k \leq m \\j,k>0}} \Re\big(b_{j,k}^\alpha z_1^j \bar{z}_1^k z_\alpha \big)\\
+ O(\vert z_n \vert \vert z \vert+ \vert z_* \vert^2 \vert z \vert + \vert z_* \vert \vert z_1 \vert^{m+1}
+ \vert z_1 \vert^{2m+1} )
\end{multline}
A change in the normal variable $z_n$ to absorb the pluriharmonic terms here i.e., $z_1^k, \bar{z}_1^k, z_\alpha ^2$ as well as $z_1^k z_\alpha, \bar{z}_1^k \bar{z}_\alpha$, 
can be done according to the following standard change of coordinates $\phi_4$ given by
\begin{equation*}
\begin{split}
z_j&=t_j \; \; (1\leq j\leq n-1), \\
z_n&=t_n- P_1(t_1,\ldots,t_{n-1})
\end{split}
\end{equation*}
where 
\[
P_1(t_1, \ldots , t_{n-1}) =\sum\limits_{k=2}^{2m}a_{k0} t_1^k - \sum\limits_{\alpha=2}^{n-1} \sum\limits_{k=1}^{m} a_k^{\alpha} t_\alpha t_1^k-\sum\limits_{\alpha=2}^{n-1} c_{\alpha}t_\alpha^2
\]
with coefficients that are smooth functions of $\zeta$.

\medskip

Finally, just as we absorbed into the simplest pure term $\Re z_n$ that occurs in the Taylor expansion, other pure terms not divisible by this variable in the last step, we may 
absorb into (some among) the simplest of non-pluriharmonic (and non-harmonic) monomials occurring there namely, $\vert z_\alpha \vert ^2$ where $2\leq \alpha \leq n-1$ other harmonic (but non-pluriharmonic) terms of degree at least two which are not divisible by them. Let us do this to those of weight at most one, remaining in (2.3) rewritten in the $t$-coordinates, 
which are of the form $\bar{t}_1^j t_\alpha$ by applying the transform $\phi_5$ given by
\begin{align*}
t_1&=w_1,\ t_n=w_n,\\
t_\alpha &=w_\alpha - P_2(w_1) \; \; (2\leq\alpha \leq n-1)
\end{align*}
where
\[
\shoveleft P_2(w_1)=\sum\limits_{k=1}^{m} b_k^\alpha w_1^k 
\]
with coefficients smooth in $\zeta$, as before (since all these coefficients are simply the derivatives of some order of the smooth defining function $r$ evaluated at $\zeta$). 
We then have the sought for simplification of the Taylor expansion. 

\medskip

Now, suppose that we already have the reduced form holding at one boundary point say the origin, then (since no further normalization would be required) $\Phi_0$ may be taken to be the identity. Note that the normalizing map $\Phi_\zeta$ is not uniquely determined -- even among the class of all maps of the same form -- owing to the (only) ambiguity in the choice of the diagonalizing map $\phi_2$; however, this choice can certainly be done in a manner such that the coefficients of that unitary matrix are smooth in $\zeta$ and satisfying the `initial condition' that $\phi_2$ for the origin is the identity. Thus in all, the map $\Phi _\zeta$ is smooth in the parameter $\zeta$ with $\Phi_0$ being the identity map. This implies that the family $\Phi_\zeta(\cdot)$ is uniformly Lipschitz (where $\zeta \in U$) and converges uniformly on compact subsets of $\mathbb{C}^n$ to the identity as $\zeta$ approaches the origin as mentioned in section 2.1. Next, the consequence of the simpler fact that the map $\Phi(\zeta,0)=\Phi_\zeta(0)$ is Lipschitz in a neighbourhood of the origin, to our setting is that 
\[
\vert g^\nu(0) \vert \leq C_1 \vert w'^\nu \vert
\]
with $C_1$ independent of $\nu$. Now
\[
F^\nu(0)= B^\nu \circ g^\nu (0)= \Big((\tau_1^\nu)^{-1}(g^\nu(0))_1,\ldots,(\tau_{n-1}^\nu)^{-1}(g^\nu(0))_{n-1},
(\tau _n^\nu)^{-1}(g^\nu(0))_n \Big).
\]
From the fact that $\tau^\nu_1 \gtrsim \gamma _\nu^{1/2}$ and $\tau^\nu_j=\gamma_\nu^{1/2}$ for $2\leq j \leq n-1$ while $\tau^\nu_n=\gamma _\nu$, we get that 
\[
\Big\vert \big(B^\nu \circ g^\nu(0)\big) \Big\vert \leq C_2 \vert w'^{\nu} \vert/ \gamma_\nu
\]
with $C_2$ again independent of $\nu$.
Now, since $f$ is Lipschitz upto $M$ and $\delta_\nu \lesssim \gamma_\nu$ we have
\begin{align*}
\vert w'^{\nu} \vert & \leq \textrm {dist} ( w'^{\nu},p'^{\nu}) + \textrm {dist} (p'^{\nu},0') \\
& = \textrm {dist}(p'^{\nu}, \partial D') + \vert f(p^\nu) - f(0) \vert \\
& \leq C_3 (\gamma _\nu + \vert p^\nu \vert) \\
& = C_3 (\gamma _\nu + \delta _\nu) \\
& \leq C_4 \gamma _\nu 
\end{align*}
with constants independent of $\nu$ as before, implying that $\{F^\nu(0)\}$ is bounded.


\subsection{Stability of the Kobayashi metric}
For $\Omega,\Sigma,U$ with $0 \in \Sigma$, as in the previous section, recall from \cite{TT} (see also \cite{Ber2}), the $M$-metric defined for $\zeta \in U \cap \Omega$ by
\[
M_{\Omega}(\zeta,X) = \sum\limits_{k=1}^{n} \vert \big(D\Phi_\zeta(\zeta) X \big)_k \vert / \tau_k(\zeta, \epsilon(\zeta)) 
=\big \vert D( B_\zeta  \circ \Phi _\zeta )(\zeta)(X) \big\vert_{l^1}
\]
with $B_\zeta = B_\zeta^{\epsilon(\zeta)}$ where $\epsilon(\zeta)>0$ is such that $\zeta+(0, \ldots, 0,\epsilon(\zeta))$ lies on $\Sigma$ and
\[
B_\zeta^\delta (z_1, \ldots ,z_n) = \big( (\tau_1)^{-1} z_1, \ldots, (\tau_n)^{-1} z_n \big)
\]
where $\tau_1 = \tau(\zeta,\delta)$, $\tau_j = \delta^{1/2}$ for $2 \leq j \leq n-1$ and $\tau_n= \delta$. Now let us get to our scaled domains $D'_\nu=B^\nu \circ g^\nu(D')$, where we recall $B^\nu=B_{p'^\nu}$ and $g^\nu=\Phi_{w'^\nu}$ and denote by ${}^\nu \Phi_\zeta$ and ${}^\nu B_\zeta$ the normalizing map and the dilation respectively, associated to $(D'_\nu, \zeta)$; we shall drop the left superscript $\nu$ for the domain $D'$. Since all the coefficients of the polynomial automorphisms ${}^\nu \Phi_\zeta(\cdot)$ depend smoothly on the derivatives of the defining functions $r_\nu$, which converge in the $C^\infty$-topology on $U$, there exists $L>1$ such that 
\begin{equation}
1/L < \vert \det\big( D({}^\nu \Phi_\zeta(z) \big) \vert <L
\end{equation}
for all $\nu$ and $\zeta \in U$. The main result of this section is the following stability theorem for the Kobayashi metric. 
\begin{thm}
There exists a neighbourhood $U$ of the origin such that  
\[
K_{D'_\nu}(z,X) \geq C M_{D'_\nu}(z,X)
\]
for all $z\in U\cap D'_\nu$, $\nu \gg 1$ where $C$ is a positive constant independent of $\nu$.
\end{thm}
\noindent The proof of this requires two steps and follows the lines of the argument presented in \cite{TT}. The first step consists in uniformly localizing the Kobayashi metric of the scaled domains near the origin which translates to verifying the following uniform version of the attraction of analytic discs near a local plurisubharmonic peak point.
\begin{lem}
There exists a neighbourhood $V \subseteq U$ of the origin and $\delta>0$ independent of $j$ such that for any $j$ large enough and any analytic disc $f:\Delta\rightarrow D'_j$ with $f(0) = \zeta\in V \cap D'^j$, we have 
\[
f(\Delta_\delta) \subseteq U \cap D'_j 
\]
\end{lem}
\begin{proof} 
If the lemma were false, then for any neighbourhood $V \subset U$ and any $\delta>0$, there exists a sequence of analytic discs $f_j: \Delta \rightarrow  D'_j$ with $f_j(0) = \zeta^j\in V \cap D'_j $, converging to the origin but $f_j(\Delta_\delta) \nsubseteq U$. Consider 
\[ 
\tilde{f}_j=S_j^{-1} \circ f_j: \Delta \rightarrow D'_j
\] 
where $S_j=B^j \circ \Phi_{w'^j}$ are the scaling maps. Then $\tilde{\zeta}^j = \tilde{f}_j(0)=S_j^{-1}(\zeta^j)$ converges to the origin. Indeed, the family $ S_j^{-1} = \Phi_{w^j}^{-1} \circ (B^j)^{-1} $ is equicontinuous at the origin since their derivatives are
\[
D(\Phi_{w^j}^{-1}) \circ (B^j)^{-1}
\]
with $D (\Phi_{w^j}^{-1})$ being bounded in a neighbourhood of the origin and $(B^j)^{-1}$ converging to the zero map. Now, $0\in\partial D'$ being a plurisubharmonic peak point by \cite{Ch2}, the simplest version of the attraction property of analytic discs (see for instance lemma 2.1.1 in \cite{G2}), gives for all $j \gg 1$ that
\[
\tilde{f}_j(\Delta_\gamma)\subseteq U \cap D'
\] 
for some $\gamma\in (0,1)$. Rescale the $\tilde{f}_j$'s, so that we may take $\gamma=1$.  Now recall the sharper version of the attraction property given by lemma 3.6 of \cite{TT} namely,
\begin{lem}
Let $W \Subset U$ be a neighbourhood of the origin. There exist constants $\alpha,A \in [0,1]$ and $K \geq 1$ such that for any analytic disc $f : \Delta_N \to U$ that satisfies $M_{D'}(f(t), f'(t)) \leq A$, $f(0) \in W$ and 
$K^{N-1} \epsilon(f(0))< \alpha$ also satisfies,
\[
\overline{f(\Delta_N)} \subset Q \Big( f(0), K^N \epsilon(f(0)) \Big)
\] 
\end{lem}
We intend to apply this lemma to the analytic discs $\tilde{g}_j(t)=\tilde{f}_j(r_0 t)$ for which we only need to verify 
\[
M_{D'} \big( \tilde{g}_j (t),\tilde{g}_j'(t) \big) \leq A 
\]
since $\tilde{g}_j(0)=\tilde{f}_j(0)$ converges to the origin.
Then 
\begin{eqnarray*}
M_{D'}\big(\tilde{g}_j(t),\tilde{g}_j '(t)\big) &=& M_{D'}\big(\tilde{f}_j(r_0 t),r_0 \tilde{f}_j'(r_0 t)\big)\\
&=& r_0 M_{D'}\big(\tilde{f}_j(r_0 t), \tilde{f}_j'(r_0 t)\big)\\
&\leq& r_0 C_5 K_{D'}\big( \tilde{f}_j(r_0 t), \tilde{f}_j'(r_0 t) \big)
\end{eqnarray*}
by theorem 3.10 of \cite{TT}, where $C_5$ is a positive constant independent of $j$. At the last step, we use the fact that $ \tilde{f}_j$'s map $\Delta$ into so small a neighbourhood of the origin where $K_{D'}\approx M_{D'}$. As the Kobayashi metric decreases under holomorphic mappings, we have
\begin{equation*}
r_0 C_5 K_{D'}\big( \tilde{f}_j(r_0t), \tilde{f}_j'(r_0t) \big) \leq r_0 C_5 K_{\Delta}\big( t, \partial/\partial t \big)
\end{equation*}
So, if we choose $r_0$ such that 
\begin{equation*}
r_0 C_5 \Big( {\displaystyle\sup_{ |t| \leq r_0}} K _{\Delta} (t, \partial/\partial t) \Big) \leq A,
\end{equation*} 
we will have completed the required verification of the hypothesis for the discs $g_j: \Delta\rightarrow D' \cap U$ and the aforementioned lemma gives
\begin{equation}
\tilde{f}_j(\Delta_{r_0})= \tilde{g}_j(\Delta)  \subset Q \big( \tilde{g}_j(0), \epsilon( \tilde{g}_j(0)) \big) = Q \big( \tilde{f}_j(0), \epsilon(\tilde{f}_j(0)) \big).
\end{equation}
At this point observe also that $r_0$ does not depend on $\delta$. Next, fix any compact subdisc $K \subset \Delta_{r_0}$ and note by (2.5) that,
\[
\tilde{f}_j(K) \subset Q\big( \tilde{f}_j(0),\epsilon ( \tilde{f}_j(0) ) \big)
\]
and subsequently,
\begin{align}
(S_j \circ \tilde{f}_j)(K) &\subset S_j \big( Q( \tilde{f}_j(0)),  \epsilon ( \tilde{f}_j(0)) \big) \nonumber\\
&= ( B^j \circ \Phi_{w'^j} ) \circ \big( B_{\tilde{\zeta}^j} \circ \Phi_{\tilde{\zeta'}^j} \big)^{-1}
(\Delta^n) \nonumber \\
&= \Big( B_{p'^j} \circ \big( \Phi_{w'^j} \circ \Phi_{\tilde{\zeta'}^j}^{-1} \big) \circ B_{\tilde{\zeta}^j}^{-1} \Big) (\Delta^n)
\end{align}
where $\tilde{\zeta}'^j = \tilde{\zeta}^j + (0, \ldots,0 , \epsilon(\tilde{\zeta}^j))$. Now the fact that $\epsilon(\tilde{\zeta}^j) \approx \epsilon(p^j)$, as both $p^j$ and $\tilde{\zeta}^j$ converge to the origin and corollary 2.8 of \cite{Ch3} -- the fact that the distinguished radius $\tau$ at any two points in $U$ are comparable -- together imply the boundedness of the sequence $\tau(\tilde{\zeta}'^j, \epsilon(\tilde{\zeta}^j))/ \tau(w'^j, \epsilon(p^j))$, by $C_6$ say. This when combined with (2.4) and (2.6), yields the stability of our discs $f_j$ on compact subdiscs in $\Delta$ i.e., we have for a positive constant $C_K>C_6L^2$ (which may depend on $K$ but not on $j$), that
\begin{align*}
f_j(K)= (S_j \circ\tilde{f_j})(K)\subseteq \Delta_{\sqrt{C_K}} \times \ldots \times \Delta_{\sqrt{C_K}} \times \Delta_{C_K}
\end{align*}
Hence by Montel's theorem, $f_j$'s converge to a map $ f_{\infty}: \Delta_{r_0} \rightarrow \overline{D'}_{\infty}$, where $f_{\infty}(0)= 0 \in \partial D'_{\infty}$. If $\phi$ is a local plurisubharmonic peak function at the origin, then the maximum principle applied to the subharmonic function $\psi = \phi \circ f_\infty$ implies that $f_{\infty} \equiv$ constant, since $\phi$ peaks precisely at one point. This contradicts our assumption that $ f_j(\Delta_\delta) \nsubseteq U $ (as soon as $\delta$ is taken smaller than $r_0$) and completes the proof of the lemma.
\end{proof}
\begin{rem}
The same line of argument  gives proposition 2.1 by replacing the discs by balls in $\mathbb{C}^m$, covering any given compact $K \subset \omega$ by balls and using the engulfing property of the special polydics $Q$.
\end{rem}
\noindent The second and more technical step, is a quantitative form of the Schwarz lemma at the boundary which generalizes lemma 3.6 of \cite{TT} mentioned earlier. To do this we require a stable version of the engulfing property for the distinguished polydiscs $Q^\nu$ assosciated to the scaled domain $D'_\nu$. Fix a pair of neighbourhoods $0 \in W \Subset V \Subset U$.
\begin{lem}
There exist constants $\alpha,A \in [0,1]$ and $K>1$ such that for every analytic disc
\[
f: \Delta_N \to V\cap D'_\nu
\]
that satisfies $M_{D'^\nu} \big( f(t), f'(t) \big) <A$, $f(0) \in W$ and $K^{N-1} \epsilon(f(0)) < \alpha$ also satisfies,
\[
\overline{f(\Delta_N)} \subset Q^\nu \Big( f(0), K^N \epsilon\big( f(0) \big) \Big).
\]
\end{lem}
\noindent Recall the definition of $A_l(z)$ and $\tau(z,\delta)$ from section 2.1. Let $A^\nu _l (z)$ and $\tau^\nu(z,\delta)$ denote the corresponding quantities for the domain $D'_\nu$ and $A^\infty_l(z)$ and $\tau^\infty(z, \delta)$ that for $D'_\infty$. Then it is clear that $A_l ^\nu(z) \to A_l^\infty(z)$ and consequently $\tau^\nu(z,\delta) \to \tau^\infty(z,\delta)$, both convergences being uniform for $z$ in $V$. Let $\delta>0$ be given. Let 
\[
M=\sup \big\{ \big( \tau^\nu(z,\delta) \big)^l : z \in V, \; 2\leq l \leq 2m',\; \nu \geq 1 \big\}
\]
and $0< \epsilon < \delta/M$. Then there exists by the uniform convergence of the $A$'s, an $N_\delta$ independent of $z \in V$ such that for all $2\leq l \leq 2m$ and all $\nu > N_\delta$,
\begin{equation}
A_l^\infty(z) - \epsilon < A^\nu_l(z) <A_l^\infty(z) + \epsilon
\end{equation}
The right inequality gives for all $2 \leq l \leq 2m$ that
\[
\big( \tau^\infty (z,\delta)\big)^l A_l^\nu(z) < \big( \tau^\infty(z, \delta) \big)^l A_l^\infty(z) + \big( \tau^\infty (z, \delta) \big)^l \epsilon  < \delta + \delta 
\]
by the definition of $\tau^\infty$ and $\epsilon$ . The definition of $\tau^\nu$ now makes this read as
\begin{equation}
\tau^\infty(z,\delta) < 2^{1/2} \; \tau^\nu(z, \delta) 
\end{equation}
for all $\nu > N_\delta$ and $z\in V$. Meanwhile the left inequality at (2.7) similarly gives,
\begin{align*}
\big( \tau^\nu(z,\delta)\big)^l A^\infty_l(z) &< \big(\tau^\nu(z,\delta)\big)^l A^\nu(z) + \big( \tau^\nu(z, \delta) \big)^l \epsilon < 2\delta 
\end{align*}
which implies $\tau^\nu(z,\delta) < 2^{1/2} \tau^\infty(z, \delta)$. Combining this with corollary 2.8 of \cite{Ch3} applied to $\tau^\infty$, we get for $z\in Q^\infty( z',\delta)$ and all $\nu$ large that 
\begin{equation}
1/C \; \tau^\nu(z',\delta)  <  \tau^\nu(z, \delta) <  C \tau^\nu(z', \delta)
\end{equation}
for some $C>0$, independent of $z'$ and $\nu$. This will lead to the following uniform engulfing property of these polydiscs.
\begin{lem}
There exists a positive constant $C$ such that for all $\nu$ large, $z'\in U$  and $z'' \in Q^\nu(z', \delta)$ we have
\begin{align*}
Q^\nu( z'', \delta) &\subset Q^\nu(z', C\delta) \\
Q^\nu(z' , \delta) &\subset Q^\nu(z'', C\delta)
\end{align*}
\end{lem}
\begin{proof} 
Let us recall what $Q^\nu$ is and rewrite  for instance the first statement as
\[
\big( {}^\nu B_{z''}^{\delta} \circ {}^\nu \Phi_{z''} \big)^{-1} \big( \Delta^n \big) \subset \big( {}^\nu B_{z'}^{C \delta} \circ {}^\nu \Phi_{z'} \big)^{-1} \big( \Delta^n \big)
\]
which is equivalent to saying that the map
\begin{equation}
{}^\nu B_{z'}^{C \delta} \circ {}^\nu \Phi_{z'} \circ ({}^\nu\Phi_{z''})^{-1} \circ ({}^\nu B^{\delta}_{z''})^{-1}
\end{equation}
maps the unit polydisc into itself. Now recall from (2.4), that the sequence of Jacobians of ${}^\nu\Phi_{z'} \circ {}^\nu\Phi_{z''}^{-1}$ is uniformly bounded above by $L^2$. Unravelling the definition of the scalings ${}^\nu B_{z'}^{C \delta}$ and ${}^\nu B_{z''}^\delta$, then shows that the map defined in (2.10) carries the unit polydisc into the polydisc of polyradius given by
\begin{equation}
\big( L^2 \tau^\nu(z'', \delta)/ \tau^\nu(z', C \delta), L^2 \delta^{1/2}/(C\delta)^{1/2}, \ldots, L^2 \delta^{1/2}/(C\delta)^{1/2}, L^2 \delta/ C\delta \big)
\end{equation}
Now the uniform comparability of the $\tau$'s obtained in (2.9) gives $C_7>0$ independent of $\nu$ and $z'$, such that 
\[
\tau^\nu(z'', \delta) < C_7 \tau^\nu(z', \delta)
\]
whereas looking at the definition of the $\tau$ directly gives 
\[
\tau^\nu (z', C\delta) > C^{1/2m} \tau^\nu (z', \delta)
\]
Combining these two we see that the first component at (2.11) is bounded above by $L^2 C_7/ C^{1/2m}$. Therefore, choosing $C$ to be such that $C^{1/2} > L^2$ and $C^{1/2m} > L^{2}C_7$ ensures that the map at (2.10) leaves the unit polydisc invariant, completing the verification of the lemma. 
\end{proof}

Before passing, let us record one more stable estimate concerning the $\tau$'s to be of use later, namely 
\begin{equation}
\tau^\nu(z, \delta) \leq C_8 \delta^{1/2m'}
\end{equation}
for $C_8>0$ (independent of $\nu$), which follows as in section 2.1 from the fact that the $A^\nu _{2m'}(z)$'s are uniformly 
bounded below by a positive constant on a neighbourhood of the origin (which as usual we may assume to be $U$).

\medskip

Next, observing that all the constants in the calculations in the proof of lemma 3.6 of \cite{TT}  are stable
owing to (2.9), lemma 2.8 and the stability of the estimates on the various coefficients of ${}^\nu\Phi$'s given in lemma 3.4 of \cite{TT} -- because all these coefficients are nothing but the derivatives of some order (less than $2m'$) of the defining functions of the scaled domains, which converge in the $C^\infty$-topology on $U$-- lemma 2.7 follows as well. 

\begin{proof}[Proof of theorem 2.3]
If the theorem were false, there would exist a subsequence $\nu_j$ of $\mathbb{N}$ and $\zeta^j \in D'^{\nu_j}$ converging to the origin such that 
\[
K_{D'_{\nu_j}} \Big( \zeta^j, X \big/ M_{D'_{\nu_j}}( \zeta^j, X) \Big) < 1/j^2
\]
which after re-indexing the $D'^{\nu_j}$'s as $D'^j$ and using lemma 2.4 reads as : \medskip\\
There exists a sequence $\zeta^j \in  U\cap D'^j$ with $\zeta^j$ converging to the origin and a sequence of analytic discs 
\[ 
f_j: \Delta \to U \cap D'_j
\]
such that 
\[
f_j(0) = \zeta^j \; \; \text{and} \; \; f'_j(0) = R_j \Big( X \big / M_{D'_j}(\zeta^j,X) \Big)
\]
with $R_j \geq j^2$. \medskip\\
The idea as in \cite{Ber2} and \cite{TT} now consists of scaling these analytic discs appropriately to enlarge their domains to discs of growing radius while ensuring their normality also, so that taking their limit produces an entire curve lying inside $D'_\infty$ to contradict its Brody hyperbolicity. To work this out, define $M_j(t) : \overline{\Delta}_{1/2} \to \mathbb{R}^{+}$ by
\[
M_j(t)=M_{D'_j}\big( f_j(t), f'_j(t) \big)
\]
Then 
\begin{align*}
M_j(0)=M_{D'_j} \big(f_j(0), f'_j(0)\big) &=M_{D'_j} \Big( \zeta^j , R_j X \big / M_{D'_j}(\zeta^j, X) \Big)\\
&= R_j \geq j^2
\end{align*}
Recall the following lemma from \cite{Ber2}. 
\begin{lem}
Let $(X,d)$ be a complete metric space and let $M:X \to \mathbb{R}^{+}$ be a locally bounded function. Then for all $\sigma>0$ and for all $u \in X$ satisfying $M(u) >0$, there exists $v \in X$ such that
\begin{enumerate}
\item[(i)] $d(u,v) \leq 2/ \sigma M(u)$
\item[(ii)] $M(v) \geq M(u)$
\item[(iii)] $M(x) \leq 2M(v)$ if $d(x,v) \leq 1/ \sigma M(v)$
\end{enumerate}
\end{lem}
\noindent Apply this lemma to $M_j(t)$ on $\overline{\Delta} _{1/2}$ with $u=0$ and $\sigma=1/j$, to get $a_j \in \overline{\Delta}_{1/2}$ such that $\vert a_j \vert \leq 2 j /M_j(0)$ and 
\begin{equation}
M_j(a_j) \geq M_j(0) \geq j^2
\end{equation}
and furthermore,
\begin{equation}
M_j(t) \leq 2M_j(a_j) \; \; \text{on} \;\; \Delta \big( a_j, j/M_j(a_j) \big)
\end{equation}
(for all $j$ big enough so that the discs above lie inside $\Delta_{1/2}$). Now, scaling the unit disc with respect to the points $a_j$ which approach the origin (since the contention is that the derivatives of the $f_j$'s near the origin, in the $M$-metric, are blowing up rapidly), we get the discs $\Delta_j = s_j(\Delta)$ where $s_j= b^j \circ \phi_{a_j}$ with $\phi_{a_j}$ the translation that transfers $a_j$ to the origin and $b^j$ being the map that dilates by a factor of $2A^{-1}M_j(a_j)$ with $A$ the constant of lemma 2.7. The scaled analytic discs $g_j : \Delta_j \to U \cap D'_j$ are then given by
\begin{equation*}
g_j(t) = (f_j\circ s_j^{-1})(t)= f_j(a_j + c_jt),
\end{equation*}
where $c_j = A/2M_j(a_j)$. Note that
\[
g'_j(t) = c_j f'_j(a_j + c_j t)
\]
which implies
\begin{align*}
M_{D'^j} \big( g_j(t), g'_j(t) \big) &= c_j M_j\Big( f_j(a_j+ c_jt), f'_j(a_j + c_j t) \Big) \\
&=c_j M_j(a_j + c_j t )
\end{align*}
Now note that $ a_j + c_j t$ lies in $\Delta\big( a_j, j/M_j(a_j) \big)$ for all $\vert t \vert <j$ as $0<A<1$
and therefore by (2.13), (2.14) and the definition of $c_j$, we have for $t \in \Delta_j$ that
\begin{align}
M_{D'^j} \big( g_j(t), g'_j(t) \big) &= c_j M_j(a_j+c_jt) \nonumber\\
&< c_j\big(2 M_j(a_j) \big) \nonumber\\
&= A
\end{align}
We wish to apply lemma 2.7 to these $g_j$'s. First we must that verify their centres lie close enough to the origin; to that end write
\[
\vert g_j(0) \vert \leq \vert f_j(a_j) - f_j(0) \vert + \vert f_j(0)\vert 
\]
and note that the first term on the right can be made arbitrarily small by the equicontinuity of the $f_j$'s which map into the bounded neighbourhood $U$. After passing to a subsequence if necessary to ensure that $\epsilon(g_j(0)) \leq \alpha/ K^{j-1}$, we have with (2.15), verified all the criteria of lemma 2.7 and therefore
\begin{equation}
g_j(\Delta_N) \subset Q^j \Big( g_j(0), K^N \epsilon( g_j(0) ) \Big)
\end{equation}
for all $j \geq N$. Let $\eta_j=g_j(0)$ and $\eta'_j = \eta_j + (0, \ldots,0, \epsilon_j)$ with $\epsilon_j >0 $ such that $r_j(\eta'_j)=0$ and note that 
\begin{equation*}
\eta_j \in Q^j \big( \eta'_j, C\epsilon_j \big)
\end{equation*}
for a uniform constant $C \geq 1$ since $\epsilon_j \approx r_j(\eta_j)$, uniformly in $j$. Consequently using lemma 2.8, (2.16) becomes
\begin{equation}
g_j(\Delta_N) \subset Q^j \Big( \eta'_j, CK^N\epsilon _j \Big)
\end{equation}
for all $j \geq N$. Now if we let 
\[
h_j = {}^j B_{\eta_j} \circ {}^j \Phi_{\eta'_j} \circ g_j : \Delta_j \to \tilde{S}_j(U \cap D')
\]
where $\tilde{S}_j$ is the map
\[
{}^j B_{\eta_j} \circ {}^j \Phi_{\eta'_j} \circ B_{p'^j} \circ \Phi_{w'^j}
\]
then what (2.17) translates to, for the map $h_j$ is that 
\[
h_j (\Delta_N) \subset \Delta_{(CK^N)^{1/2}} \times \ldots \times \Delta_{(CK^N)^{1/2}} \times \Delta_{CK^N}
\]
for all $N \leq j$, exactly as in the proof of lemma 2.4. Also, note that the domains $\tilde{S}_j(U \cap D')$ after passing to a subsequence if necessary, converge to a domain $\tilde{D}'_\infty$ of the same form as $D'_\infty$, by the same argument as in section 2.1 together with (2.12). Thus Montel's theorem, a diagonal sequence argument and an application of the maximum principle to the limit, gives rise to an entire curve
\[
h: \mathbb{C} \to \tilde{D}'_\infty. 
\]
Indeed, note that 
\[
h_j(0) = \big( {}^j B_{\eta_j} \circ {}^j \Phi_{\eta'_j} \big)(\eta_j) =(0, \ldots,0,-1)
\]
for all $j$ and so $h(0)=(0, \ldots, 0, -1)$, which lies in $\tilde{D}'_\infty$ all of whose boundary points -- including the point at infinity -- are peak points. Now, to check that $h$ is nonconstant, we examine the sequence of their derivatives at the origin. We write $X_j$ for $g'_j(0) = c_jf'_j(a_j)$ (which tend to $0$ but are bounded below in the $M$-metric).
\begin{align}
\big\vert h'_j(0) \big\vert_{l^1} &= \big\vert \big( {}^j B_{\eta'_j} \circ D ({}^j \Phi_{\eta'_j}) \big) (\eta_j) (X_j) \big\vert_{l^1}\\
&\geq    \big\vert \big(  {}^j B_{\eta_j} \circ D ({}^j\Phi_{\eta_j})  \big)(\eta_j) (X_j) \big \vert_{l^1} \nonumber\\
&= M_{D'^j} \big( g_j(0), X_j \big) \nonumber\\
&=c_jM_j(a_j) =A/2 \nonumber
\end{align}
where the lower bound here is a consequence of (2.9) applied to $\tau^j$ for the points $\eta_j$ and $\eta'_j$ and the fact that the family
\[
D({}^\nu \Phi_{z'}) \circ D( {}^\mu \Phi_{z''}^{-1})
\]
is uniformly bounded below in norm. Passing to the limit in (2.18), now gives
\[
\vert h'(0) \vert _{l^1} \geq C(A/2)  >0
\]
and we reach the contradiction mentioned earlier namely to the Brody hyperbolicity of $\partial \tilde{D}'_\infty$ (see for instance, lemma 3.8 of \cite{TT}). 
\end{proof}
\noindent Finally, let us note the consequence of theorem (2.3) in the form that we shall make use of
\begin{cor}
There is neighbourhood $U$ of the origin and a positive constant $C$ such that
\[
K_{D'^\nu}(z,X) \geq C \; \vert X \vert / \big( {\rm dist}(z, \partial D'^\nu) \big)^{1/2m'} 
\]
for all $z \in U \cap D'_\nu$ and all $\nu \gg 1$. 
\end{cor}
\noindent This comes from the facts that $\epsilon_\nu (z) \approx \textrm{dist}(z, \partial D'^\nu)$ and for some $C_9>0$ we have for all $\nu$ that
\[
M_{D'^\nu}(z,X) \geq C_9 \vert X \vert_{l^1} / (\epsilon^\nu(z) )^{1/2m'}
\]
which in turn follows from (2.12).

\subsection{Uniform H\"older continuity of the scaled maps near the origin}
In this section, we recall from \cite{CS}, the arguments which show that the family of scaled maps is uniformly H\"older continuous upto the boundary -- we already know that each scaled map is H\"older continuous upto the origin.
\begin{thm}
There exist positive constants $r,C$ such that 
\[
\vert F^\nu(z') -F^\nu(z'') \vert \leq C\vert z' - z'' \vert^{1/2m'}
\]
for any $z', z'' \in \bar{D}_\nu \cap B(0,r)$ and all large $\nu$.
\end{thm} 
\noindent By the previous section we may assume that $\{F^\nu(0)\}$ converges to a (finite) boundary point $q \in \partial D'_{\infty}$. Recall that the origin (respectively the $\Re z_n$ axis) is a common boundary point (respectively, the common normal at the origin) for all the scaled domains. At this special boundary point, we first show that the family of scaled mappings is equicontinuous -- we already know that this holds on compact subsets of $D_{\infty}$-- or equivalently that, given any neighbourhood of $q$, there exists a small ball about the origin, such that every scaled map carries the piece of its domain intercepted by this ball into that given neighbourhood. Then, the uniform lower bound on the Kobayashi metric near $q$ will sharpen the uniform boundary distance decreasing property to uniform H\"older continuity, of the scaled maps. In particular then, the family ${\{ F^\nu \}} \cup {\{F\}}$ is equicontinuous near the origin, upto the respective boundary.
\begin{lem}
For any $\epsilon>0$ there exists $\delta>0$, such that $\vert F^\nu(z)-q\vert <\epsilon$ for any $z\in D_\nu \cap B(0, \delta)$ and all large $\nu$.
\end{lem}
\noindent {\it Proof.}
Suppose to obtain a contradiction, that the assertion were false. Then there exists $\epsilon_0>0$ and a sequence $a_\nu \in D_\nu$ such that $a_\nu \to 0$ and $\vert F^\nu(a_\nu)-q\vert \geq\epsilon_0$. 
Since every $F^\nu$ is continuous upto the boundary $\partial D_\nu$, we can also choose a sequence of points $b_\nu \in D_\nu$ lying on the common inner normal to $D_\nu$ at the origin i.e., $b_\nu=(0,- \beta _\nu)$ with $ \beta _\nu>0$, such that $b_\nu \to 0$ and $\vert F^\nu(b_\nu)-q\vert \to 0$. 
Let $s_\nu=\vert a_\nu-b_\nu \vert$. It is not difficult to see that there exists a constant $C>0$ such that for all $\nu$, there is a smooth path $\gamma_\nu : [0,3s_\nu] \to D_\nu$ with the following properties:
\begin{enumerate}
\item[(i)] $\gamma _\nu(0)=a_\nu$, $\gamma_\nu(3s_\nu)=b_\nu$
\item[(ii)] ${\rm dist}(\gamma _\nu(t), \partial D_\nu) \geq Ct$, for $t\in[0,s_\nu]$,\\
${\rm dist} (\gamma_\nu(t), \partial D_\nu)\geq Cs_\nu$, for $t\in[s_\nu, 2s_\nu]$,\\ 
${\rm dist}(\gamma_\nu(t), \partial D_\nu)\geq C(3s_\nu - t)$, for $t\in[2s_\nu, 3s_\nu]$, 
\item[(iii)] $\vert d\gamma_\nu(t) / dt \vert \leq C$, for $t\in[0,s_\nu]$.
\end{enumerate}

\noindent By corollary (2.10) there is a positive constant $C$ such that for any $w \in D'_\nu \cap B(q,2\alpha)$ for some $\alpha >0$, and any vector $X \in \mathbb{C}^n$ the lower bound
\begin{equation}
K _{D'_\nu}(w,X) \geq C \; \textrm {dist}(w, \partial D'_\nu)^{-1/2m'} \vert X \vert
\end{equation}
holds for all large $\nu$. Let $\eta=\rm{min} \{ \alpha/4, \epsilon _0/4 \}$. Since $F^\nu(b_\nu)$ lies in $B(q,\alpha)$ for $\nu$ large enough, we can choose $t_\nu\in[0,3s_\nu]$ such that $F^\nu\circ \gamma_\nu(t_\nu)\in \partial B(q,2\eta)$ and $F^\nu \circ \gamma _\nu((t_\nu,3s_\nu])$ is contained in $B(q,2\eta)$. From section 2.1 we know that
\begin{equation}
{\rm dist}(F^\nu(z), \partial D'_\nu) \leq C(R) \; {\rm dist}(z, \partial D_\nu)
\end{equation}
for any $R>0$ and $z\in D_\nu \cap B(0,R)$ with $F^\nu(z) \in D'_\nu \cap (B(q,2\alpha))$. Fix $r>0$ and let $z\in D_\nu \cap B(0,r)$ be such that $F^\nu(z) \in D'_\nu \cap B(q,2\tau)$. Since the Kobayashi metric is decreasing under holomorphic mappings, we get from (2.19) and (2.20) that 
\begin{align*}
K_{D_\nu}(z,X) &\geq K_{D'_\nu} (F^\nu(z), dF^\nu(z)X)\\
&\geq C \; {\rm dist}(F^\nu(z),\partial D'_\nu)^{-1/2m'} \vert dF^\nu(z)X \vert \\
&\geq C \; {\rm dist}(z,\partial D_\nu)^{-1/2m'} \vert dF^\nu(z)X \vert.
\end{align*}

\noindent On the other hand it can be seen that for all $\nu$,
\[
K_{D_\nu} (z,X) \leq \vert X\vert/ \textrm {dist}(z, \partial D_\nu)
\]
and this implies the uniform estimate 
\begin{equation}
\vert dF^\nu (z) \vert \leq C \textrm {dist} (z, \partial D_\nu)^{-1 + 1/2m'}
\end{equation}
for all large $\nu$ and $z\in D_\nu \cap B(0,r)$ such that $F^\nu(z) \in D_\nu \cap B(p,2\alpha)$. Therefore,
\begin{align*}
\vert F^\nu(\gamma^\nu(t_\nu)) - F^\nu(b_\nu) \vert &\leq \int_{t_\nu}^{3s_\nu} \vert dF^\nu(\gamma^\nu(t)) \vert
\vert d\gamma_\nu/ dt\vert dt \\
&\leq C \int_{t_\nu}^{3s_\nu} \textrm {dist} (\gamma_\nu(t), \partial D_\nu)^{-1+1/2m'} dt \\
& \leq Cs_\nu^{1/2m'} \to 0
\end{align*}
as $\nu \to \infty$, which is a contradiction.
\medskip
\begin{proof}[{Proof of theorem 2.11}]
It was just shown that ${F^\nu(z)}$ lies in $B(q,2\alpha)$ for any $z\in \bar{D}_\nu \cap B(0, \delta)$ if $\delta>0$ is chosen small enough. Hence (2.21) holds for any $w=F^\nu(z)$ and a similar integration argument as above, then gives the estimate asserted in the theorem with a uniform constant. 
 \end{proof}

\subsection{Compactness of $f^{-1}(0)$}
Recall that to establish theorem 1.1, we were to prove that $f^{-1}(0)$ is compact in $M$. Suppose to obtain a contradiction that this were false. Then the intersection
\[
f^{-1}(0) \cap M \cap \partial B(0,\epsilon) \neq \phi.
\]
for all $\epsilon>0$ small. Since the scaled mappings differ from $f$ by a biholomorphic change of coordinates on the domain and the target, the same holds for them as well, i.e.,
\[
(F^\nu)^{-1}(F^\nu(0))\cap \partial D_\nu \cap \partial B(0,\epsilon)\neq \phi.
\]
Let us show that this property passes to the limit as well, i.e.,
\[
F^{-1}(q) \cap \partial D_{\infty} \cap \partial B(0,\epsilon)  \neq \phi
\]
for all $\epsilon>0$ small -- so small that Theorem 2.11 holds for $r=\epsilon$. Let $a^\nu$ be in  $(F^\nu)^{-1}(F^\nu(0))\cap \partial D_\nu \cap \partial B(0,\epsilon)$ and $a^\nu \to a\in \partial D_{\infty} \cap \partial B(0,\epsilon)$. Let $b$ be a point in $D_{\infty}$ near $a$. Since $D_\nu \to D_{\infty}$, there exists an integer $N$ such that for all $\nu > N$, $b$ lies in $D^\nu$ and
\[
\vert F(a) - q \vert \leq \vert F(a) - F(b) \vert + \vert F(b)- F^\nu (b) \vert + \vert F^\nu (b) - F^\nu (a^\nu) \vert + \vert F^\nu(a^\nu) - q \vert
\]
By taking $\nu$ large enough, the second term can be made arbitrarily small by the convergence of $F^\nu$ at $b$ and the last term as well by the convergence 
of $F^\nu (a^\nu)=F^\nu(0)$ to $q$. The 
equicontinuity of the family ${\{ F^\nu \}} \cup {\{ F \}}$ given by theorem 2.11, 
then assures that the same can be done with the third and the first terms, by taking $b$ sufficiently 
close to $a$ and $\nu$ larger if necessary. Thus, $F^{-1}(q)$ 
is not compact in any neighbourhood of the origin in $\partial D_{\infty}$. 

\medskip

On the other hand, by starting from the standpoint of $F$ being a holomorphic mapping between algebraic domains that extends continuously upto a 
boundary piece $\Sigma$ of $\partial D_\infty$, we may apply the theorem of Webster to establish the algebraicity of $F$. 
Thus $F$ extends as an analytic set and by the invariance property of Segre varieties (see \cite{CP}) as a locally finite to one holomorphic map near $0 \in \pa D_{\infty}$. 
Contradiction.

\medskip

To make this work, we only have to show that $F$ does not map an open piece of $\pa D_{\infty}$ into the weakly pseudoconvex points on $\pa D'_{\infty}$. Let us denote by 
$w(\Gamma)$ the set of all weakly pseudoconvex points of a given smooth hypersurface $\Gamma$. Since 
\[
\partial D_\infty = \big\{ z \in \mathbb{C}^n \; :\; 2\Re z_n + P('z, ' \bar{z})=0 \big\}
\]
it follows that 
\[
w(\partial D_\infty) = \big( Z \times \mathbb{C} \big) \cap \partial D_\infty
\]
where $Z$ is the real algebraic variety in $\mathbb{C}^{n-1}$ defined by the vanishing of the determinant of the complex Hessian of $P$ in $\mathbb{C}^{n-1}$.  The finite type assumption on $\partial D_\infty$ implies that $w(\partial D_\infty)$ is a real algebraic set of dimension at most $2n-2$. Recall that 
\[
F(\Sigma) \subset \partial D'_\infty = \big \{ z \in\mathbb{C}^n : 2\Re z_n + Q_{2m'}(z_1,\bar{z}_1) + \vert z_2 \vert^2 + \ldots + \vert z_{n-1} \vert^2 =0 \big \},
\] 
and denote $Q_{2m'}$ by $Q$ for brevity.

\begin{prop}
The map  $F$ extends to an algebraic map in $\mathbb{C}^n$.
\end{prop}
\begin{proof}[{Proof.}]
Suppose now that there exists a strictly pseudoconvex point $a \in \partial D_\infty$ such that $F(a)$ is also a strictly pseudoconvex point in $\partial D'_{\infty}$. Then by \cite{PT}, $F$ is a smooth CR-diffeomorphism near $a$ and extends locally biholomorphically across $a$ by the reflection principle in \cite{P1}. Then Webster's theorem \cite{W} assures us that $F$ is algebraic. So, we may as well let $\Sigma$ be the set of strictly pseudoconvex points, assume that $F$ maps it into $w(\partial D'_{\infty})$ and argue only to obtain a contradiction. As before note that 
\[   
w(\partial D'_{\infty}) = \big\{ z\in \mathbb{C}^n \; : \; \Delta Q(z_1,\bar{z}_1) =0 \big\} \cap \partial D'_{\infty}.
\]
Then $w(\partial D'_\infty)$ admits a semi-analytic stratification by real analytic manifolds of dimension $2n-2$ and $2n-3$. Using the specific form of $\partial D'_\infty$, this can be explicitly described as follows. Let us denote by $V=V(\Delta Q)$ the real algebraic variety in $\mathbb{C}$ defined by the polynomial $\Delta Q$. The set of singular points $\textrm{Sng}(V) $, near which $V$ may fail to be a smooth curve is finite. 
Let $a\in w(\partial D'_\infty)$ be such that $\pi_1(a) \in \textrm{Reg} (V)$ where $\pi_1$ is the natural surjection onto the $z_1$-axis. Then $V$ is a smooth curve near $\pi_1(a)$ and after a biholomorphic change of coordinates, we may assume that in its vicinity $V$ coincides with $\{ \Im z_1 =0\}$. Note that the fibre of $\pi_1$ over $a$ in $w(\partial D'_\infty)$ namely $\pi_1^{-1}(\pi_1(a)) \cap w(\partial D'_\infty)$, is given by 
\[
\Big\{ z \in \mathbb{C}^n  :  2\Re z_n + \vert z_2 \vert^2 + \ldots + \vert z_{n-1} \vert^2 = -Q( \pi_1(a), \overline{\pi_1(a)} ) \Big\} 
\cap \Big\{ z \in \mathbb{C}^n : z_1 = \pi_1(a) \Big\}
\]
which is evidently equivalent to $\partial \mathbb{B}^{n-1}\subset \mathbb{C}^{n-1}$. This description of the fibres of $\pi_1$ restricted to $w(\partial D'_\infty)$ persists in a neighbourhood of $a$ and then the real analytic strata $S_{2n-2} \subset w(\partial D'_\infty)$ of dimension $2n-2$ is locally biholomorphic to 
$\partial \mathbb{B}^{n-1} \times \{ \Im z_1 =0 \}$. The complement of $S_{2n-2}$ in $w(\partial D'_\infty)$, which we will denote by $S_{2n-3}$ is locally biholomorphic to $\partial \mathbb{B}^{n-1} \times \textrm{Sng}(V)$ and this evidently has dimension $2n-3$. Observe also that $S_{2n-2}$ is generic since its complex tangent space has dimension $n-2$ as a complex vector space. Now, if 
\[
F(\Sigma) \cap S_{2n-2} \neq \phi
\]
then by the continuity of $F$, there is an open piece of $\Sigma$ that is mapped by $F$ into $S_{2n-2}$. Denote this piece by $\tilde{\Sigma}$. We may assume that $0 \in \tilde{\Sigma}$ and that $F(0)=0 \in S_{2n-2}$. Using an idea from Lemma 3.2 in \cite{DP2}, let $L$ be a 2 dimensional complex plane that intersects $w(\partial D'_\infty)$ in a totally real submanifold of real dimension 2 -- this is possible by the genericity of $S_{2n-2}$ -- and this also holds for all translates $L_{a'}$ of $L$ passing through $a'$ in a sufficiently small neighbourhood $U'$ of the origin. We may therefore find a non-negative, strictly plurisubharmonic function $\phi_{a'}$ on $U'$ that vanishes on 
\[
S_{a'}=L'_{a'} \cap w(\partial D'_{\infty}) \cap U'.
\]
Indeed by a change of coordinates, we may assume that
\[ 
L= \textrm{span} _{\mathbb{C}}  \left< \partial / \partial  z_1, \partial / \partial z_n \right> =\{ z_2= \ldots =z_{n-1}=0 \} 
\]
and then
\[
\phi_{a'}= \vert z_2 - a'_2 \vert^2 + \ldots + \vert z_{n-1}- a'_{n-1} \vert^2 + \vert \Im z_1 \vert^2 + (r'_{\infty}(z,\bar{z}))^2
\]
furnishes an example.
By the continuity of $F$ and the openness of $\tilde{\Sigma}$ in $\partial D_\infty$, we can pick $b\in D_\infty$ so near the origin that $F(b) \in U'$ and $\partial A_b \subset \tilde{\Sigma}$ where 
\[
A_b=\{ z\in \mathbb{C}^n : z_n=b_n \} \cap F^{-1}(L_{F(b)})
\]
Since the pull-back of an analytic set under a holomorphic map is again analytic of no lesser dimension, $A_b$ is a positive dimensional analytic set. Also, $b \in A_b$ and $ F( \partial A_b) \subset S_{F(b)}$. Therefore, $\psi_b =\phi_{F(b)} \circ F $ is a non-negative, plurisubharmonic funtion on $A_b$ that vanishes on $\partial A_b$. By the maximum  principle $\psi_b\equiv 0$ on $A_b$ which implies that $F$ maps all of $A_b$ into $S_{F(b)}$. Since $F$ maps $D_\infty$ into $D' _\infty$, this is a contradiction.\medskip\\
To finish, note that the the remaining possibility is $F(\Sigma) \subset S_{2n-3}$. The above argument can be repeated in this case as well -- we will only need to replace $\vert \Im z_1\vert^2$ as a subharmonic function vanishing on a curve-segment of $\textrm{Reg}(V)$ by $\vert p(z_1) \vert^2 $ where $p$ is a holomorphic polynomial that vanishes on the finite set $\textrm{Sng}(V)$.
\end{proof}


\section{Proof of theorem 1.3}
\noindent We begin with the scaling template for $(D,D',f)$ in section 2.1, maintaining as far as possible the notations therein. For clarity and completeness, let us briefly describe the scaling of $D'$ which is simpler this time as $M'$ is strongly pseudoconvex. As before, let $p^\nu=('0, -\delta_\nu) \in D$ and note that $p'^\nu = f(p^\nu)$ converges to the origin which is  a strongly pseudoconvex point on $\partial D'$. Let $w^\nu \in \partial D'$ be such that 
\[
\vert w^\nu - p'^\nu \vert = \textrm{dist} (p'^\nu, \partial D') = \gamma_\nu
\]
Furthermore, since $\partial D'$ is strongly pseudoconvex near the origin, we may choose a strongly plurisubharmonic function in a neighbourhood of the origin that serves as a defining function for $D'$. Arguing as in Section 2.1, it follows that $f$ preserves the distance to the boundary, i.e.,
\[
\delta_\nu \approx \gamma_\nu
\]
for $\nu \gg 1$. For each $w'^\nu$, lemma 2.2 in \cite{P2} provides a degree two polynomial automorphism of $\mathbb{C}^n$ that firstly, transfers $w '^\nu$ and the normal to $\partial D'$ there, to the origin and the  $\Re z_n$-axis respectively and secondly, ensures that the second order terms in the Taylor expansion of the defining function $r' \circ (g^\nu)^{-1}$ of the domain $g^\nu (D')$ constitute a hermitian form that coincides with the standard one i.e., $\vert z_1 \vert^2 + \ldots + \vert z_{n-1} \vert^2 $, upon restriction to the complex tangent space. Define the dilations
\[
B^\nu ('z, z_n)= ( \gamma ^{-1/2}_\nu \; {'z}, \gamma _\nu^{-1}z_n  )
\]
and note that the scaled domains $D'_\nu = (B^\nu \circ g^\nu)  (D')$ are defined by 
\[
\gamma _\nu^{-1} r'_\nu (z) = 2 \Re(g^\nu(z))_n + \vert (g^\nu(z))_1 \vert^2+\ldots+\vert (g^\nu(z))_{n-1}\vert^2 + O(\gamma _\nu ^{1/2}).
\]
These converge in the Hausdorff metric to 
\[
D'_\infty = \mathbb{H} = \{ z\in \mathbb{C}^n : 2\Re z_n + \vert z_1 \vert^2 + \ldots + \vert z_{n-1} \vert^2 \ <0 \},
\]
which is the unbounded manifestation of the ball, in view of the fact that the $g^\nu$'s converge uniformly on compact subsets of $\mathbb{C}^n$ to the identity map. \medskip\\
Standard arguments as in \cite{P2} show that the scaled maps 
\[
F^\nu = B^\nu \circ g^\nu \circ f \circ (T^\nu)^{-1}
\]
converge to a map $F: D_\infty \to \overline{\mathbb{H}}$. If some point $z_0$ of $D_\infty$ is sent by $F$ to $w_0 \in \partial\mathbb{H} \cup \{\infty\}$, then composing $F$ with a local peak function at $w_0$, we get a function holomorphic on a neighbourhood of $z_0$ and peaking precisely at $z_0$. By the maximum principle, $F(z)\equiv w_0$. However, $F^\nu('0,-1)= B^\nu \circ g^\nu \circ f \circ T^\nu('0,-1)= B^\nu \circ g^\nu(f(p^\nu)) = B^\nu('0, -\gamma_\nu)=('0,-1)$. Hence, $F('0,-1)=('0,-1)\neq w_0$. Thus, $F$ maps $D_\infty$ into $D'_\infty$  and is again as in section 2, a non-degenerate, locally proper map extending continuously upto the boundary in a neighbourhood of the origin. By theorem 2.1 in \cite{CP}, $F$ extends holomorphically across the origin. By composing with a suitable automorphism of the ball, we may also assume $F(0)=0$. Then, $\Re(F_n(z))$ is a pluriharmonic function that is negative on $D_\infty$ and attains a maximum at the origin. So, by the Hopf lemma, we must have $\alpha = \partial(\Re F_n)/\partial x_n(0) > 0$ which combined with the fact that $DF$ preserves the complex tangent space and thereby the complex normal at the origin (to the hypersurfaces $\partial D_\infty$ and $\partial D'_\infty$, which themselves correspond under $F$ near the origin), implies that 
\[
F_n(z)= \alpha z_n + g(z)
\]
for some holomorphic function $g$ with $g(z)=o(\vert z \vert)$. Now, let us compare the two defining functions 
for $D_\infty$, near the origin:
\[
2 \Re (F_n(z)) + \vert F_1(z) \vert^2+ \ldots + \vert F_{n-1}(z)\vert^2 
= h(z,\bar{z})(2\Re z_n + P('z,'\bar{z}))
\]
for a non-vanishing real analytic function $h(z,\bar{z})$. 
Contemplate a weighted homogeneous expansion of the above equation with respect to the weight $(1/m_n, \ldots , 1/m_2)$ given by the multitype of $\partial D_\infty$. Note firstly that on the left, pluriharmonic terms arise precisely from $\Re(F_n(z))$. Next, note that the lowest possible weight for any term on the right is one and the non-pluriharmonic component of this weight is $h(0)P('z,'\bar{z})$. What this means for the left, is that each $F_j$ must expand as
\[
F_j(z)=P_j('z) + \text{ (terms of weight} > 1/2)
\]
where each $P_j$ is either weighted homogeneous of weight $1/2$ or identically zero and
\begin{equation}
h(0)P('z,'\bar{z}) = \vert P_1('z) \vert^2 + \ldots + \vert P_{n-1} ('z) \vert^2
\end{equation}
Clearly, all the $P_j$'s cannot be zero as $P$ is non-zero. In fact, the finite type character of $\partial D_\infty$ forces all of them to be non-zero, as follows. After a rearrangement if necessary, assume that 
$P_j \in \mathbb{C}['z] = \mathbb{C}[z_1,\ldots,z_{n-1}]$ is non-zero precisely when $1\leq j \leq m \leq n-1$. Then the common zero set $V \subset \mathbb{C}^{n-1}$, of these $P_j$'s gives rise to the complex analytic variety $i(V)$ in $\partial  
D_\infty$ where $i: \mathbb{C}^{n-1} \to \mathbb{C}^n$ is the natural inclusion. The finite type constraint compels this variety and needless to say $V$, to be discrete. Furthermore, the weighted homogeneity reduces it to \{0\} : If $(z_1, \ldots ,z_{n-1})$ is a non-trivial zero of the $P_j$'s and $t\in \mathbb{C}$, then 
\[
P_j(e^{t/m_n}z_1, \ldots , e^{t/m_2} z_{n-1}) = e^{t/2} P_j(z_1,\ldots,z_{n-1}) =0
\]
for all $1 \leq j \leq m$ and so the entire curve defined by
\[
\gamma(t) = (e^{t/m_n}z_1,\ldots,e^{t/m_2}z_{n-1})
\]
lies inside $V$. Now, consider the ideal $I$ generated by these polynomials $P_1, \ldots, P_m$, which is a $\mathbb{C}$-algebra whose transcendence degree cannot exceed $m \leq n-1$. On the other hand, by the Nullstellensatz, for a large integer $N$ the algebraically independent monomials, $z_1^N, \ldots, z_{n-1}^N$ must all lie in $I$, forcing its transcendence degree over $\mathbb{C}$ and hence $m$ to equal $n-1$. The upshot therefore, is that $P$ is the squared norm of a weighted homogeneous polynomial endomorphism $\hat{P}$ of $\mathbb{C}^{n-1}$ with $\hat{P}^{-1}(0)=0$. 
Now, put $z_2= \ldots =z_{n-1}=0$ in (3.1). This cannot reduce the left hand side there to zero, otherwise the $z_1$-axis will lie inside the zero set of $P$ and hence in $\partial D_\infty$. What this means for the right hand side of (3.1), is that terms involving $z_1$ (and $\bar{z}_1$) must occur. Since all polynomials therein are homogeneous of same weight, we conclude that terms involving $z_1,\bar{z}_1$ there, must only be of the form $c\vert z_1 \vert^{m_n}$ for some $c>0$. Needless to say, the same holds for all the other variables $z_j$ for $2\leq z_j \leq n-1$, as well and we have the last statement of theorem 1.3, namely
\[
P('z, '\bar{z})= c_1 \vert z_1 \vert ^{m_n} + c_2 \vert z_2 \vert^{m_{n-1}} + \ldots + c_{n-1} \vert z_{n-1} \vert^{m_2} + \text{mixed terms}
\]
with all $c_j$'s being positive and the mixed terms comprising of weight one monomials in $'z, '\bar{z}$ each of which is annihilated by at least one of the natural 
quotient maps $\mathbb{C}['z,'\bar{z}] \to \mathbb{C}['z, '\bar{z}]/(z_j \bar{z}_k)$, $1 \leq j, k \leq n-1$ where $j \not= k$.


\section {Proof of theorem 1.4}
\begin{proof}
\noindent As before by the lower semi-continuity of rank, $\partial D'$ is of rank at least $n-2$ in a neighbourhood $\Gamma' \subset \partial D'$ of $p'$ which we may assume also to be of finite type and pseudoconvex, consequently regular. Let us apply Theorem A of \cite{BS} to the proper holomorphic correspondence $f^{-1} : D' \to D$. The hypothesis on $cl_{f^{-1}}(\Gamma')$ there holds, since $\partial D$ is globally regular. Therefore by that theorem, $f^{-1}$ extends continuously upto $\Gamma'$ as a proper correspondence.  However, we are not certain of the splitting of $f^{-1}$ near $p$ into branches -- for that $p$ we will have to lie away from the branch locus of $f^{-1}$. Nevertheless, there  exist neighbourhoods $U'$ of $p'$ and $U$ of $p$ and a (local) correspondence
\[
f^{-1}_{loc} : U' \cap D' \to U \cap D,
\] 
extending continuously upto the boundary such that the graph of $f^{-1}_{loc}$ is contained in that of $f^{-1}$ and $cl_{f^{-1}_{loc}}(p') = \{ p \}$ where the last condition comes from the fact that $f$ is finite to one upto the boundary, by that theorem again. Assume that both $p=0$ and $p'=0$ and choose a sequence $p'^\nu=('0,-\delta_\nu)\in D' \cap U'$ on the inner normal approaching the origin. Since $f:D \to D'$ is proper and $0 \in cl_f (0)$ there exists a sequence $p^\nu \in D$ with $p^\nu \to 0$ such that $f(p^\nu)=p'^\nu$. Moreover, by the continuity of $f^{-1}_{loc}$ upto the boundary and the condition $cl_{f^{-1}_{loc}} (0) = \{0 \}$, we may assume after shrinking $U,U'$ if necessary, that $p^\nu \in D \cap U$ with $f^{-1}_{loc}(p'^\nu)=\{ p^\nu \}$. Now scale $D$ with respect to $\{ p^\nu \}$ and $D'$ with respect to $\{p'^\nu\}$ -- to scale $D'$, we only consider the dilations 
\[
T^\nu(z_1,\ldots,z_n)=\big((\tau_1^\nu)^{-1}z_1,(\tau_2^\nu)^{-1}, \ldots, (\tau_n^\nu)^{-1}z_n \big)
\]
where $\tau_1^\nu=\tau(0,\delta_\nu)$, $\tau_j^\nu=\delta_\nu^{1/2}$ for $2\leq j \leq n-1$ and $\tau_n^\nu = \delta_\nu$, while for $D$ we use the composition $B^\nu \circ g^\nu$ where $g^\nu$ and $B^\nu$ are as in section 3. As before, the limiting domains for $D_\nu=(B^\nu \circ g^\nu)(U \cap D)$ and $D_\nu'= T^\nu(U' \cap D')$ are the ball $\mathbb{B}^n$ and 
\[
D'_\infty = \big\{z\in \mathbb{C}^n \; : \; 2 \Re z_n + Q_{2m'}(z_1, \bar{z}_1) + \vert z_2\vert^2 +\ldots + \vert z_{n-1} \vert^2 <0 \big \},
\]
respectively, where this time $Q_{2m'}$ is homogeneous of degree $2m'$ and coincides with the polynomial of this degree in the (homogeneous) Taylor expansion of the defining function $r'$ for $\partial D'$ near the origin. Normality of the scaled mappings 
\[
F^\nu = T^\nu \circ f \circ (B^\nu \circ g^\nu)^{-1}
\]
follows as before by proposition 2.1. But section 2 can no longer guarantee nondegeneracy of a limit map, owing to the lack of a clear cut boundary distance conservation property of $f$. However, the existence of the inverse as a proper correspondence paves the way for a different approach. To begin with, recall from \cite{KP}, the notion of normality for correspondences and theorem 3 therein, the version of Montel's theorem for proper holomorphic correspondences with varying domains and ranges. Let $K_\mu$ be an exhaustion of $D'_\infty$ by compact subsets containing $('0,-1)$. To establish the normality of scaled correspondences 
\[
(f_{loc}^{-1})^\nu = (B^\nu \circ g^\nu) \circ f_{loc}^{-1} \circ (T^\nu)^{-1},
\]
it suffices by theorem 3 of \cite{KP} to show that $(f_{loc}^{-1})^\nu(K_\mu) \Subset D_\infty$ for each $\mu \in \mathbb{N}$. To this end, fix a $K_\mu$, let $(\tilde{f}_{loc}^{-1})^\nu$ denote the correspondence $f^{-1}_{loc} \circ (T^\nu)^{-1}$ and note that $(\tilde{f}^{-1}_{loc})^\nu (K_\mu)$ is connected (since any of its components must contain $p^\nu$). Now (viewing the origin in $\partial D$ as the sequence $p^\nu$ that approaches it from the interior of $D$), we recall the Schwarz lemma for correspondences given by theorem 1.2 of \cite{V}, which assures us that the images of $K_\mu$ under these correspondences $(\tilde{f}^{-1}_{loc})^\nu$, will be contained in Kobayashi balls about $p^\nu$ of a fixed size, i.e., there exists $R>0$ independent of $\nu$ such that 
\[
(\tilde{f}_{loc}^{-1})^\nu  (K_\mu) \subset B_D^K(p^\nu,R),
\]
where $B_\Omega^K(p,R)$ denotes the Kobayashi ball centered at $p\in \Omega$ of radius $R$, for any given domain $\Omega$. Then, since biholomorphisms preserve Kobayashi balls,
\[
(f^{-1}_{loc})^\nu  (K_\mu) \subset (B^\nu \circ g^\nu)((\tilde{f}^{-1}_{loc})^\nu (K_\mu)) \subset (B^\nu \circ g^\nu)(B_D^K(p_\nu,R)) = B_{D_\nu}^K \big(('0,-1),R \big)
\]
while $B_{D_\nu}^K \big( ('0,-1),R \big) \subset B_{\mathbb{B}^n}^K \big(('0,-1),R + 1 \big)$ for all $\nu$ large by lemma 4.4 of \cite{MV}, giving the stability of the images of 
$K_\mu$ under scaling.  Noting that $\big(('0,-1),('0,-1)\big) \in \textrm{Graph}(f^{-1}_{loc})^\nu$ for all $\nu$, we conclude that $\{ (f^{-1}_{loc})^\nu \}$ must admit a 
subsequence that converges to a correspondence that is inverse to $F$, where $F:\mathbb{B}^n \to D'_\infty$ is a limit of $\{ F^\nu \}$ (which means that the composition of these 
correspondences \textbf{contains} the graph of the identity). By \cite{BS}, both $F$ and the inverse correspondence $F^{-1}$ extend continuously upto the respective boundaries 
and in particular $F$ is finite to one on the boundary. 
Then, $F$ is smooth upto the boundary by \cite{BC}, preserves boundaries and then is  
algebraic again as in section 2.6. and thereafter by \cite{BB2} extends holomorphically past the boundary. 
Now recall that the determinant of the Levi forms are related as 
\begin{equation*}
\lambda_{\mathbb{B}^n} (z) = \lambda_{D'_\infty}(F(z)) \vert J_F(z) \vert^2
\end{equation*}
for all $z\in \partial \mathbb{B}^n$, which readily gives the strict pseudoconvexity of $\partial D'_\infty$. In particular therefore, $Q_{2m'}(z_1, \bar{z}_1)$ must be $\vert z_1 \vert^2$ which gives the strict pseudoconvexity of $ p'$ in $\partial D'$. 
\end{proof}
\medskip
\begin{proof}[Proof of corollary 1.5]
By the proof of Theorem 1.1 in section 2, we have that $f^{-1}(f(p))$ is compact in $M$ and as described in \cite{BC}, it is possible to choose neighbourhoods $U$ of $p$ and $U'$ of $p'$ in $\mathbb{C}^n$ such that if $D$ and $D'$ are the pseudoconvex sides of $U \cap M$ and $U' \cap M'$ respectively then $f$ extends to be a proper map from $D$ into $D'$, putting us in the situation of theorem 1.4.
\end{proof}


\begin{thebibliography}{Ber1}

\bibitem{BR} M. S. Baouendi, L. P. Rothschild: 
\textit{Germs of CR maps between real analytic hypersurfaces}, Invent. Math. \textbf{93} (1988), 481--500.

\bibitem{BJT} M. S. Baouendi, H. Jacobowitz, F. Treves:
\textit{On the analyticity of CR mappings}, Ann. of Math. (2) \textbf{122} (1985), 365--400.

\bibitem{B} E. Bedford:
\textit{Proper holomorphic mappings}, Bull. Amer. Math. Soc. (N.S) \textbf{10} (1984), 157--175.

\bibitem{BB1} E. Bedford, E. Barletta:
\textit{Existence of proper mappings from domains in $\mbb C^2$},
Indiana Univ. Math. J. \textbf{39} (1990), no. 2, 315--338.

\bibitem{BB2} E.Bedford, S.Bell:
\textit{Extension of proper holomorphic mappings past the boundary}, Manuscripta Math. \textbf{50} (1985), 1-10.

\bibitem{BF} E. Bedford, J. E. Fornaess:
\textit{Local extension of CR functions from weakly pseudoconvex boundaries}, Michigan Math. J. \textbf{25} (1978), 259--262.

\bibitem{Be1} S. Bell:
\textit{CR maps between hypersurfaces in $\mbb C^n$}, in \textit{Several Complex Variables and Complex Geometry, Part 1 (Santa Cruz, Calif., 1989)}, Proc. Sympos. Pure Math. 
\textbf{52}, Part 1, Amer. Math. Soc., Providence, (1991), 13--22. 

\bibitem{Be2} S. Bell:
\textit{Local regularity of CR homeomorphisms}, Duke Math. J. \textbf{57} (1988), 295--300.

\bibitem{BC} S. Bell, D. Catlin:
\textit{Regularity of CR mappings}, Math. Z. \textbf{199} (1988), 357--368.

\bibitem{BN} S. Bell, R, Narasimhan:
\textit{Proper holomorphic mappings of complex spaces}, in \textit{Several Complex Variables} (W. Barth, R. Narasimhan Eds.), Encyclopaedia Math. Sci. \textbf{69}, 
Springer-Verlag, Berlin (1990).

\bibitem{Ber1} F. Berteloot:
\textit{Attraction des disques analytiques et continuit\'{e} h\"{o}ld\'{e}rienne dápplications holomorphes propres}, in \textit{Topics in Complex Analysis (Warsaw, 1992)}, 
Banach Center Publ. \textbf{31}, Polish Acad. Sci., Warsaw, (1995), 91--98.

\bibitem{Ber2} F. Berteloot:
\textit{Principle de Bloch et estimations de la metrique de Kobayashi des domains in $\mbb C^2$}, J. Geom. Anal. \textbf{1} (2003), 29--37.

\bibitem{BS} F. Berteloot, A. Sukhov: 
\textit{On the continuous extension of holomorphic correspondences}, Ann. Scuola Norm. Sup. Pisa Cl. Sci. (4)  \textbf{24}  (1997),  no. 4, 747--676.

\bibitem{BerC} F. Berteloot, G. Coeur\'{e} :
\textit{Domaines de $\mathbb{C}^2$, pseudoconvexes et de type fini ayant un groupe non compact d'automorphismes}, Ann. Inst. Fourier (Grenoble)  41  (1991),  no. 1, 77–86.

\bibitem{C1} D. Catlin:
\textit{Estimates of invariant metrics on pseudoconvex domains of dimension two}, Math. Z. \textbf{200} (1989), 429--466.

\bibitem{C3} D. Catlin: 
\textit{Boundary invariants of pseudoconvex domains},  Ann. of Math. (2)  120  (1984),  no. 3, 529–586. 

\bibitem{Ch2} S. Cho:
\textit{A lower bound on the Kobayashi metric near a point of finite type in $\mbb C^n$},  J. Geom. Anal. \textbf{2}  (1992),  no. 4, 317--325. 

\bibitem{Ch3} S. Cho:
\textit{Boundary behavior of the Bergman kernel function on some pseudoconvex domains in $\mbb C^n$}, Trans. Amer. Math. Soc.  \textbf{345}  (1994),  no. 2, 803--817. 

\bibitem{CP} B. Coupet, S. Pinchuk:
\textit{Holomorphic equivalence problem for weighted homogeneous rigid domains in $\mathbb{C}^{n+1}$}
Complex analysis in modern mathematics (Russian),  57–70, FAZIS, Moscow, 2001.

\bibitem{CS} B. Coupet, A. Sukhov:
\textit{On CR mappings between pseudoconvex hypersurfaces of finite type in $\mbb C^2$},  Duke Math. J. \textbf{88}  (1997),  no. 2, 281--304.

\bibitem{CGS} B. Coupet, H. Gaussier, A. Sukhov:
\textit{Regularity of CR maps between convex hypersurfaces of finite type}, Proc. Amer. Math. Soc. \textbf{127}
(1999), no. 11, 3191--3200.

\bibitem{CPS1} B. Coupet, S. Pinchuk, A. Sukhov:
\textit{On boundary rigidity and regularity of holomorphic mappings}, Internat. J. Math. \textbf{7} 91996), 617--643.

\bibitem{D} J. D'Angelo:
\textit{Real hypersurfaces, orders of contact, and applications}, Ann. of Math. \textbf{115} (1982), 615--637.

\bibitem{DF} K. Diederich, J. E. Fornaess:
\textit{Pseudoconvex domains: An example with nontrivial Nebenh\"{u}lle}, Math. Ann. \textbf{225} (1977), 275-292.

\bibitem{DF1} K. Diederich, J. E. Fornaess:
\textit{Proper holomorphic images of strictly pseudoconvex domains}, Math. Ann.  \textbf{259} (1982), no. 2, 279--286.

\bibitem{DH} K. Diederich, G. Herbort:
\textit{Pseudoconvex domains of semiregular type}, Contributions to complex analysis and analytic geometry,  127–161, Aspects Math., E26, Vieweg, Braunschweig, 1994.

\bibitem{DP1} K. Diederich, S. Pinchuk:
\textit{Proper holomorphic maps in dimension 2 extend}, Indiana Univ. Math. J. \textbf{44} (1995), 1089--1126.

\bibitem{DP2} K. Diederich, S. Pinchuk:
\textit{Regularity of continuous CR maps in arbitrary dimension}, Michigan Math. J. \textbf{51} (2003), 111--140.

\bibitem{DW} K. Diederich, S. Webster:
\textit{A reflection principle for degenerate real hypersurfaces}, Duke Math. J. \textbf{47} (1980), 835--845.

\bibitem{FS} J. E. Fornaess, N. Sibony:
\textit{Construction of P.S.H. functions on weakly pseudoconvex domains},  Duke Math. J.  \textbf{58}  (1989),  no. 3, 633--655.

\bibitem{F} F. Forstneric:
\textit{Proper holomorphic mappings: A survey}, in \textit{Several Complex Variables}, Proceedings of the special year at the Mittag-Leffler Institute (Ed. J. E. Fornaess) 
(1993), 297--363, Princeton University Press, Princeton, NJ.

\bibitem{G} H. Gaussier:
\textit{Smoothness of Cauchy Riemann maps for a class of real hypersurfaces}, Publ. Mat.  \textbf{45}  (2001),  no. 1, 79--94.

\bibitem{G2}H. Gaussier:
\textit{Tautness and complete hyperbolicity of domains in $\mathbb{C}^n$},  Proc. Amer. Math. Soc.  127  (1999),  no. 1, 105–116.

\bibitem{Her} G. Herbort:
\textit{On the invariant differential metrics near pseudoconvex boundary points where the Levi form has corank one},  Nagoya Math. J.  \textbf{130}  (1993), 25--54. 

\bibitem{KP} Klingenberg, W.; Pinchuk, S.:
\textit{Normal families of proper holomorphic correspondences},  Math. Z.  207  (1991), \textbf{1}, 91–96. 

\bibitem{Kr} S. Krantz:
\textit{Convexity in complex analysis},  \textit{Several Complex Variables and Complex Geometry, Part 1 (Santa Cruz, Calif., 1989)}, Proc. Sympos. Pure Math.
\textbf{52}, Part 1, Amer. Math. Soc., Providence, (1991), 119--137.

\bibitem{MV} P. Mahajan, K. Verma:
\textit{Some aspects of the Kobayashi and Carath\'{e}odory metrics on pseudoconvex domains}, To appear in J. Geom. Anal.

\bibitem{M3} J. Merker:
\textit{On envelopes of holomorphy of domains covered by Levi-flat hats and the reflection principle},  Ann. Inst. Fourier (Grenoble)  \textbf{52}  (2002),  no. 5, 
1443--1523. 

\bibitem{P1} S. Pinchuk:
\textit{On the analytic continuation of holomorphic mappings}, Math. USSR Sb. \textbf{27}, 375--392.

\bibitem{P2} S. Pinchuk:
\textit{Holomorphic inequivalence of certain classes of domains in $\mathbb{C}^{n}$},  Mat. Sb. (N.S.)  111(153)  (1980),\textbf{1}, 67–94, 159.

\bibitem{PT} S. Pinchuk, Sh. Tsyganov:
\textit{Smoothness of CR mappings of strictly pseudoconvex hypersurfaces}, Izv. Akad. Nauk. SSSR Ser. Mat. \textbf{53} (1989), 1120--1129.

\bibitem{Su} A. Sukhov:
\textit{On boundary behaviour of holomorphic mappings}, (Russian), Mat. Sb. \textbf{185} (1994), 131--142; English transl. 
in Russian Acad. Sci. Sb. Math. \textbf{83} (1995), 471--483.

\bibitem{TT} Do D. Thai, Ninh V. Thu:
\textit{Characterization of domains in $\mbb C^n$ by their noncompact automorphism groups}, Nagoya Math. J. \textbf{196} (2009), 135--160.

\bibitem{Tr} J. M. Tr\'{e}preau:
\textit{Sur le prolongement holomorphe des fonctions CR definies sur une hypersurface re\'{e}lle de classe $C^2$ dans $\mbb C^n$}, 
Invent. Math. \textbf{43} (1977), 53--68.

\bibitem{V} K. Verma:
\textit{A Schwarz lemma for correspondences and applications}, Publ. Mat.  \textbf{47}  (2003),  no. 2, 373--387,

\bibitem{W} S. Webster:
\textit{On the mapping problem for algebraic real hypersurfaces}, Invent. Math. \textbf{43} (1977), 53--68.

\bibitem{Yu1} Yu, Ji Ye:
\textit{Weighted boundary limits of the generalized Kobayashi-Royden metrics on weakly pseudoconvex domains}, Trans. Amer. Math. Soc.  \textbf{347}  (1995),  no. 2, 587–614. 

\bibitem{Yu2} Yu, Ji Ye:
\textit{Peak functions on weakly pseudoconvex domains},  Indiana Univ. Math. J.  43  (1994),  no. 4, 1271–1295.

\end{thebibliography}
\end{document}